\documentclass[smallextended, draft, numbook, envcountsame]{svjour3}
\usepackage{appendix}
\usepackage{framed}
\usepackage{amsfonts}
\usepackage{latexsym}
\usepackage{lmodern}
\usepackage[T1]{fontenc}
\usepackage[latin9]{inputenc}
\usepackage[letterpaper]{geometry}
\usepackage{color}
\usepackage{multirow}
\usepackage{float}
\usepackage{array}
\usepackage[english]{babel}
\usepackage{verbatim}
\usepackage{prettyref}
\usepackage{amsmath}
\usepackage{amssymb}
\usepackage{stmaryrd}
\usepackage[unicode=true,pdfusetitle,
 bookmarks=true,bookmarksnumbered=false,bookmarksopen=false,
 breaklinks=false,pdfborder={0 0 1},backref=false,colorlinks=false]
 {hyperref}
\makeatletter
\usepackage{amsmath}
\usepackage{amsfonts}
\usepackage{latexsym}
\usepackage{amssymb}
\usepackage{framed}

\newcommand{\beq}{\begin{equation}}
\newcommand{\eeq}{\end{equation}}
\newcommand{\beqa}{\begin{eqnarray}}
\newcommand{\eeqa}{\end{eqnarray}}
\newcommand{\beqas}{\begin{eqnarray*}}
    \newcommand{\eeqas}{\end{eqnarray*}}
\newcommand{\bi}{\begin{itemize}}
    \newcommand{\ei}{\end{itemize}}
\newcommand{\gap}{\hspace*{2em}}

\newcommand{\vgap}{\vspace{.1in}}

\newcommand{\R}{\mathbb{R}}

\newcommand{\lam}{{\lambda}}

\newcommand{\inner}[2]{\langle #1,#2\rangle}

\newcommand{\cl}{\mathrm{cl}\,}

\newcommand{\argmin}{\mathrm{argmin}}

\newcommand{\dom}{\mathrm{dom}\,}

\newcommand{\bConv}[1]{\mbox{\rm C}\overline{\mbox{\rm onv}}\,(\Re^{#1})}

\usepackage{enumitem}
\setlist[itemize]{leftmargin=5.5mm}
\usepackage[nocompress]{cite}
\DeclareMathOperator*{\trc}{tr}
\global\long\def\argmin{\operatorname*{arg\,min}}%
\definecolor{green}{rgb}{0,0.6,0.0} % Darken the default green

\makeatother

\usepackage[labelfont=bf]{caption}

\global\long\def\r{\Re}

\global\long\def\R{\Re}

\global\long\def\pt{\partial}

\global\long\def\n{\mathbb{N}}

\global\long\def\rn{\Re^{n}}

\global\long\def\bConv{\overline{\text{Conv}}(\rn)}

\global\long\def\lam{\lambda}

\providecommand{\tabularnewline}{\\}

\begin{document}

\title{An efficient adaptive accelerated inexact proximal point method for solving linearly constrained nonconvex composite problems}

\author{
    Weiwei Kong
    \and
    Jefferson G. Melo
    \and
    Renato D.C. Monteiro
}
\institute{
    Weiwei Kong \and Renato D.C. Monteiro \at
    School of Industrial and Systems Engineering, Georgia Institute of Technology, Atlanta, GA, 30332-0205. (\email{\tt{wkong37@gatech.edu} \& \tt{monteiro@isye.gatech.edu}\rm}). The works of these authors
    were partially supported by ONR Grant N00014-18-1-2077.
    \\\\
    Jefferson G. Melo \at
    Institute of Mathematics and Statistics, Federal University of Goias, Campus II- Caixa Postal 131, CEP 74001-970, Goi\^ania-GO, Brazil. (\email{\tt{jefferson@ufg.br}\rm}). The work of this author was
    supported in part by  CNPq Grant 406975/2016-7.
}
\date{\today}

\maketitle
\begin{abstract}
This paper proposes an efficient adaptive variant of a quadratic penalty accelerated inexact proximal point (QP-AIPP) method proposed earlier by the authors. Both the QP-AIPP method and its variant solve linearly set constrained nonconvex composite optimization problems using a quadratic penalty approach where the generated penalized subproblems are solved by a variant of the underlying AIPP method. The variant, in turn, solves a given penalized subproblem by generating a sequence of proximal subproblems which are then solved by an accelerated composite gradient algorithm. The main difference between AIPP and its variant is that the proximal subproblems in the former are  always convex while the ones in the latter are not necessarily convex due to the fact that their prox parameters are chosen as aggressively as possible so as to improve efficiency. The possibly nonconvex proximal subproblems generated by the AIPP variant are also tentatively solved by a novel adaptive accelerated composite gradient algorithm based on the validity of some key convergence inequalities. As a result, the variant generates  a sequence of proximal subproblems where the stepsizes are adaptively changed according to the responses obtained from the calls to the accelerated composite gradient algorithm. Finally, numerical results are given to demonstrate the efficiency of the proposed AIPP and QP-AIPP variants.

\end{abstract}

\section{Introduction} \label{sec:intro}

This paper presents a computationally efficient variant of the quadratic penalty
accelerated inexact proximal point (QP-AIPP) method studied in \cite{WJRproxmet1}.

The QP-AIPP method of \cite{WJRproxmet1} is designed for solving the linearly--constrained nonconvex composite optimization problem
\begin{equation}\label{eq:probintroa}
\min \left\{ f(z) + h(z) : Az = b, \,  z \in \R^n \right \},
\end{equation}
where  $A:\r^n \mapsto \r^p$ is a linear operator, $b\in\R^p$, $h:\R^n \to (-\infty,\infty]$ is a closed proper convex function, 
and $f$ is a real-valued differentiable (possibly nonconvex) function whose gradient is $L$--Lipschitz and which, for some $0<m\leq L$, satisfies
\begin{equation}\label{eq:PenaltyProb-introa}
 f(u) \geq f(z)+\left\langle \nabla f(z),u-z\right\rangle -\frac{m}{2}\|u-z\|^{2} 
\quad \forall \, z,u \in \dom\, h.
\end{equation}
The QP-AIPP method solves \eqref{eq:probintroa} via a quadratic penalty method, i.e.,  a sequence of penalty subproblems of the form
\begin{align} \label{eq:pen_sub_intro}
\min \left\{f(z) + h(z) + \frac{c}{2} \|Az - b\|^2  : z \in \R^n \right \},
\end{align}
for an increasing sequence of positive penalty parameters $c$, is solved by the accelerated inexact proximal point (AIPP) method (discussed below) in which each penalty subproblem is solved using a common starting point $z_0 \in \dom h$ (i.e., a cold--start strategy is adopted).

We briefly outline the AIPP method of \cite{WJRproxmet1}. First, note that \eqref{eq:pen_sub_intro} is a special case of 
\begin{equation}\label{eq:PenProb2Introa}
\phi_* := \min \left\{ \phi(z):=g(z) + h(z)  : z \in \R^n \right \}
\end{equation}
where $g(z) := f(z) +  c \|Az-b\|^2/2$ is a function satisfying
\begin{equation} \label{eq:mMintro}
-\frac{m}{2}\|u-z\|^{2} \leq g(u)-\left[g(z)+\left\langle \nabla g(z),u-z\right\rangle \right] \leq  \frac{M}{2}\|u-z\|^{2}
\quad \forall \, z,u \in \dom\, h,
\end{equation}
where $M=L+c\|A\|^2$. In the general setting of \eqref{eq:PenProb2Introa}--\eqref{eq:mMintro},  the AIPP method generates a sequence $\{z_k\}$ using an inexact proximal point (IPP) framework
(see for  example \cite{Rock:ppa,hpe_svaiter99}), i.e., given $z_{k-1} \in \dom h$, it computes $z_{k}$
as a suitable approximate solution of the
proximal subproblem
 \begin{equation}\label{eq:penPbRegIntroa}
\min \left\{ g(z) + h(z) + \frac{1}{2\lam_k}\|z-z_{k-1}\|^2 : z \in \R^n \right \}
\end{equation}
for some prox-parameter $\lambda_k>0$.
Note that the first inequality in \eqref{eq:mMintro} implies that
the objective function of  \eqref{eq:penPbRegIntroa} is convex as long as
$\lambda_k$ is not larger than $1/m$.
The AIPP method sets $\lambda_k=1/(2m)$ for every $k$ and
uses an accelerated composite gradient (ACG) variant (see for example \cite{beck2009fast,MontSvaiter_fista,Nesterov1983})
to approximately solve \eqref{eq:penPbRegIntroa}.

Since the larger $\lam_k$ is the faster the above IPP framework converges to a desirable approximate solution,
the main goal of this paper is to develop
an  aggressive AIPP variant, and subsequently an aggressive QP-AIPP variant, which possibly chooses $\lambda_k$ substantially larger than $1/m$ despite potential
loss of convexity of \eqref{eq:penPbRegIntroa}.
An important ingredient in obtaining  this aggressive AIPP variant is the development of a relaxed ACG (R-ACG) algorithm that approximately solves \eqref{eq:penPbRegIntroa} according to a more relaxed termination criterion.
More specifically, within a reasonably number of  iterations, the algorithm:
(i) either solves the possibly nonconvex subproblem \eqref{eq:penPbRegIntroa} according to the relaxed criterion or stops with failure due to $\lam_k$ being too large; and (ii) always solves \eqref{eq:penPbRegIntroa} according to the relaxed criterion when its objective function is convex.
The aforementioned relaxed AIPP (R-AIPP) variant starts with a relatively large initial prox parameter and,
in each one of its steps, calls the R-ACG algorithm to solve the
corresponding prox subproblem.  If a key descent inequality fails, then the prox parameter $\lambda_k$ is halved, the prox center $z_{k-1}$ is maintained, and the R-ACG algorithm is invoked once again to solve the resulting prox subproblem; otherwise, the prox parameter $\lambda_k$ is preserved and $z_k$ takes the place of $z_{k-1}$.

This paper also considers a more general version of \eqref{eq:probintroa} in which the linear constraint $Az=b$ is replaced by the linear set constraint $Az\in S$, where $S\subseteq \r^p$ is a closed convex set. Clearly, when $S=\{b\}$, the more general problem reduces to \eqref{eq:probintroa}. Under the assumption that $\dom h$ is bounded and all penalty subproblems are solved by the AIPP variant using the aforementioned cold--start strategy, it turns out that the iteration complexity of the  QP-AIPP variant for finding the desired approximate solution is considerably worse than that of the QP-AIPP method of \cite{WJRproxmet1}. If, on the other hand, the QP-AIPP variant adopts the warm--start strategy in which
the R-AIPP method for solving the current penalty subproblem starts from the approximate solution found for
the previous subproblem, then the iteration complexity of this relaxed QP-AIPP (R-QP-AIPP) variant is shown to be the same as that of the QP-AIPP method of \cite{WJRproxmet1}, up to a logarithmic factor.

%solves this more general problem, albeit with a considerably worse iteration complexity compared to the QP-AIPP method in \cite{WJRproxmet1}, despite the fact that it uses a more relaxed R-AIPP variant to solve its (potentially) nonconvex prox subproblems.

%This paper also shows that: (i) the QP-AIPP variant can be applied to a more general form of \eqref{eq:probintroa} in which the linear constraint $Az=b$ is replaced with a linear set constraint $Az\in S$, where $S\subseteq \r^p$ is a closed convex set and $A:\r^n \mapsto \r^p$ is a linear operator; and (ii) the iteration complexities of the resulting variants are within a logarithmic factor of their original counterparts. In the case where $\dom h$ is bounded, the competitive iteration complexity of the QP-AIPP variant requires the use of a warm--start strategy, namely, that the last point of the current penalty subproblem is used as the starting point of next (if required) penalty subproblem. This variant differs from its original counterpart in that the QP-AIPP method of \cite{WJRproxmet1} always uses a common starting point for each penalty subproblem.

The proposed AIPP and QP-AIPP variants are compared with three state-of-the-art optimization methods on five different optimization problems. The computational results obtained show that
the variants can substantially outperform most of the competing methods on many problem instances.

{\it Related works.} 
We first discuss papers dealing with related algorithms for solving the convex version of \eqref{eq:probintroa} and other related monotone problems.
Iteration-complexity analysis of quadratic penalty methods for solving  \eqref{eq:probintroa} under the assumption that
$f$ is convex and  $h$  is a convex indicator function was first studied in \cite{LanRen2013PenMet} and further explored  in
\cite{Aybatpenalty,IterComplConicprog}. Iteration-complexity of first-order augmented Lagrangian methods for solving
the latter class of linearly constrained convex programs was studied in  \cite{AybatAugLag,LanMonteiroAugLag,ShiqiaMaAugLag16,zhaosongAugLag18,Patrascu2017,YangyangAugLag17}.
Inexact proximal point methods using  accelerated gradient algorithms to solve their prox-subproblems were previously considered in \cite{GlanPDaccel2014,YHe2,YheMoneiroNash,OliverMonteiro,MonteiroSvaiterAcceleration}
 in the setting of convex-concave saddle point problems and monotone variational inequalities.
 
We now discuss papers dealing with related algorithms for solving \eqref{eq:probintroa}
%under the assumption that $f$ is a real-valued and differentiable function whose gradient is Lipschitz on $\dom h$},
when $f$ is nonconvex and the assumptions mentioned after \eqref{eq:probintroa} hold.
%without specifically assuming convexity of the objective function.
Paper \cite{WJRproxmet1} is, up to our knowledge,  the first one to consider a proximal method with acceleration strategy for solving   \eqref{eq:probintroa}. Previous  works using acceleration strategies were concerned  with   the unconstrained problem \eqref{eq:PenProb2Introa}. Namely, \cite{nonconv_lan16}  proposed an accelerated gradient framework to solve \eqref{eq:PenProb2Introa} with better iteration complexity than the usual composite gradient method. Since then, many authors have proposed other accelerated frameworks for solving \eqref{eq:PenProb2Introa} under different assumptions on the functions $g$ and $h$
(see, for example, \cite{Aaronetal2017,Paquette2017,Ghadimi2019,Li_Lin2015,CatalystNC}).
In particular, by exploiting the lower curvature $m$,     \cite{Aaronetal2017,Paquette2017,CatalystNC} proposed some algorithms which improve the iteration-complexity bound of \cite{nonconv_lan16} in terms of the dependence on  the upper curvature $M$. Finally,  there has been a growing interest in the iteration complexity  of methods for solving optimization problems using  second order information (see, for example,  \cite{Aaronetal2017,MonteiroSvaiterNewton,NesterovSec_ord,CartToint}).

{\it Organization of the paper.} 
Subsection~\ref{sec:DefNot} provides some basic definitions and notation. Section~\ref{sec:background} begins with 
presenting some background materials 
and transitions into defining a general descent (GD) framework for solving the nonconvex 
optimization problem \eqref{eq:PenProb2Introa}.  Section~\ref{sec:Nesterov's-Method} presents and derives the complexity of an R-ACG algorithm
which attempts to solve \eqref{eq:penPbRegIntroa}  even when it is not convex.   Section~\ref{sec:AIPPmet} presents a relaxed variant 
of the AIPP method proposed in \cite{WJRproxmet1}. Section~\ref{sec:penalty} presents a relaxed variant of the QP-AIPP method proposed in \cite{WJRproxmet1}. Section~\ref{sec:numerical} presents numerical results to illustrate the efficiency of the  AIPP and QP-AIPP variants. Finally, Section~\ref{sec:concl_remarks} presents some concluding remarks.

\subsection{Basic definitions and notation \label{sec:DefNot}}

This subsection provides some basic definitions and notation used
in this paper.

The set of natural numbers is denoted by $\mathbb{N}$. The set of real numbers is denoted by $\Re$. The set of non-negative real numbers  and 
the set of positive real numbers are denoted by $\Re_+$ and $\Re_{++}$, respectively. Let $\Re^n$ denote a real valued $n$--dimension inner product space, whose inner product and its associated induced norm are denoted by $\left\langle \cdot,\cdot\right\rangle $
and $\|\cdot\|$, respectively. Let $\left\langle \cdot,\cdot\right\rangle_F $ denote the Frobenius inner product. Let $S_+^n$ denote the cone of positive semidefinite $n$--by--$n$ matrices. For $t>0$, define 
$\log^+_1(t):= \max\{\log t ,1\}$. The set of proper lower
semi-continuous convex functions defined on $\Re^n$
is denoted by $\overline{\text{Conv}}(\Re^n)$. Given a linear operator $A:\R^n \mapsto \R^p$, the operator norm of $A$ is denoted by $\|A\| := \sup\{\|Az\|/\|z\| : z\in \rn, z\neq 0\}$.

Let $\psi: \Re^n\rightarrow (-\infty,+\infty]$ be given. The effective domain of $\psi$ is denoted by
$\dom \psi:=\{x \in \Re^n: \psi (x) <\infty\}$ and $\psi$ is proper if $\dom \psi \ne \emptyset$.
%Moreover, a proper function $\psi: \Re^n\rightarrow (-\infty,+\infty]$ is $\mu$-strongly convex for some $\mu \ge 0$ if
%$$
%\psi(\alpha z+(1-\alpha) u)\leq \alpha \psi(z)+(1-\alpha)\psi(u) - \frac{\alpha(1-\alpha) \mu}{2}\|z-u\|^2
%$$
%for every $z, u \in \dom \psi$ and $\alpha \in [0,1]$.
If $\psi$ is differentiable at $\bar z \in \Re^n$, then its affine   approximation $\ell_\psi(\cdot;\bar z)$ at $\bar z$ is denoted by
\begin{equation}\label{eq:defell}
\ell_\psi(z;\bar z) :=  \psi(\bar z) + \inner{\nabla \psi(\bar z)}{z-\bar z} \quad \forall  z \in \Re^n.
\end{equation}
Also, for $\varepsilon \ge 0$,  its \emph{$\varepsilon$-subdifferential} at $z \in \dom \psi$ is denoted by
\begin{equation}\label{eq:epsubdiff}
\partial_\varepsilon \psi (z):=\left\{ v\in\Re^n: \psi(u)\geq \psi(z)+\left\langle v,u-z\right\rangle -\varepsilon,\forall u\in\Re^n\right\}.
\end{equation}
The subdifferential of $\psi$ at $z \in \dom \psi$, denoted by $\partial \psi (z)$, corresponds to  $\partial_0 \psi(z)$. 

For a given $X\subseteq \rn$, the closure of the set $X$ is denoted by $\cl X$,
 the indicator function of $X$, denoted by $\delta_X$, is defined as $\delta_X(x) = 0$ if $x\in X$ and $\delta_X(x)=\infty$ if $x\notin X$. Moreover, the normal cone of $X$ at a point $x\in X$ is denoted by
\begin{equation*}
N_{X}(x):=\{u\in\r^{n\times n}:\left\langle u,x'-x\right\rangle \leq0,\forall x'\in X\} = \pt \delta_X(x).
\end{equation*}

%?????? The proof of the following result can be found in \cite[Proposition 4.2.2]{Hiriart2}.
%
%\begin{proposition}\label{prop:transpForm}  ??????? Let $\psi:\Re^n\rightarrow (-\infty,+\infty]$, $z, \bar z \in \dom \psi$ and $v \in \Re^n$ be given and
%    assume  that $v\in\partial \psi (z)$. Then,
%    $v\in \partial_\varepsilon \psi (\bar{z})$
%    where $\varepsilon =  \psi(\bar{z})-\psi(z)-\langle v, \bar{z}-z\rangle\geq0.$
%\end{proposition}

%????? For the iteration complexity results in the sections below, given a $z_0 \in \dom \phi$ for some function $\phi$ that is bounded below, 
%we will make use of the quantity $\Delta^0_\phi:=\phi(z_0)-\phi_*$ where $\phi_* := \inf_z \phi(z)$. 

%When a solution of the problem $\min_z \phi(z)$ exists, we will also make of the quantity $d_0 := \inf\{\|z_0-z\| : \phi(z) = \phi_*\}$.
 
\section{A general descent framework}\label{sec:background}

%This section discusses a general descent (GD) scheme for approximately solving the composite nonconvex optimization problem \eqref{eq:PenProb2Introa}.

As discussed in Section \ref{sec:intro}, all the penalized subproblems (see \eqref{eq:PenaltyProb-introa}) that arise
during  the execution of the QP-AIPP method, as well as the R-QP-AIPP method, are of the form \eqref{eq:PenProb2Introa}. Hence, efficiently obtaining a solution of \eqref{eq:PenProb2Introa} is of paramount importance for both the QP-AIPP and R-QP-AIPP methods. While the QP-AIPP method uses the AIPP method to solve \eqref{eq:PenProb2Introa}, the R-QP-AIPP method uses the R-AIPP method which will be discussed in Section~\ref{sec:AIPPmet}.
The discussion of this section (as well as Section~\ref{sec:Nesterov's-Method}) will essentially pave the way towards the presentation of the R-AIPP method.

More specifically, this section presents and analyzes a GD framework for solving \eqref{eq:PenProb2Introa} that makes use of  a black box (see step~1 of the GD framework below). In addition, it describes:
the assumptions and relevant quantities underlying problem \eqref{eq:PenProb2Introa}, the notion of approximate stationary point of \eqref{eq:PenProb2Introa} adopted in this section and Section~\ref{sec:AIPPmet}, and the relationship between the GD framework and the GIPP framework of \cite{WJRproxmet1}, of which the AIPP method is an instance of.

%????The details of a particular implementation of this black box are given in Sections~\ref{sec:Nesterov's-Method} and \ref{sec:AIPPmet}, the latter of which contains the description of the R-AIPP method.

%Our effort for this and the next section is devoted to presenting a relaxed version of the AIPP method described in \cite{WJRproxmet1}, namely, the R-AIPP method, for approximately solving \eqref{eq:PenProb2Introa}, and hence the aforementioned penalized subproblems.
%More specifically, this section presents a GD scheme for solving \eqref{eq:PenProb2Introa} through
%the use of a black box (see step 1 of the GD scheme below).
%The details of a particular implementation of this black box are given in Sections~\ref{sec:Nesterov's-Method} and \ref{sec:AIPPmet}, the latter of which contains the description of the R-AIPP method. 

Our problem of interest in this section and Section \ref{sec:AIPPmet} is \eqref{eq:PenProb2Introa} which is assumed to
 satisfy the following assumptions:

\begin{itemize}[align=left]
    
\item[(A1)]$h \in \bConv$;
\item[(A2)]$g$ is a nonconvex differentiable function on $\dom h$ and there exist a scalar $M > 0$ such that
\begin{align}
  \label{eq:curvature}
  \|\nabla g(u) - \nabla g(z) \| \leq M \|u-z\| \quad \forall u,z\in \dom h;
%  \quad    g(u) \ge  \ell_g(u;z)  - \frac{m}{2}\|u-z\|^{2};
\end{align}
\item[(A3)]$\phi_*>-\infty$.

\end{itemize}
In addition, the analysis in Section~\ref{sec:AIPPmet} makes use of the quantity
\begin{equation}\label{eq:m_lower_def}
\underline{m} := \inf \left\{m\in \r_{++}: g(u) \ge  \ell_g(u;z)  - \frac{m}{2}\|u-z\|^{2} \quad \forall u,z \in \dom h \right\}, 
\end{equation}
which is positive in view of assumption (A2).
While it is generally difficult to compute the above quantity, it is well known that assumption (A2) implies that $\underline{m} \in (0,M]$. Moreover, it is shown in Proposition~\ref{prop:relaxAIPPmethod} below that the smaller $\underline m$ is, the better the iteration complexity of R-AIPP method in Section~\ref{sec:AIPPmet} becomes.
%However, the analysis of the R-AIPP method in  Section~\ref{sec:AIPPmet} allows for the possibility that an $m<M$ as in (A2) exists, and derives sharper iteration complexity bounds for if that is the case. 

It is well-known that a necessary condition for $z^*\in\dom h$ to be a local minimum of \eqref{eq:PenProb2Introa} is that $z^*$ be a stationary point of $\phi$, i.e.,
 $0\in\nabla g(z^*) + \partial h(z^*)$. A relaxation of this inclusion leads to the following definition of  an approximate stationary point of
 \eqref{eq:PenProb2Introa}: given a tolerance  $\hat\rho>0$, a pair $(\hat z, \hat v)$ is said to be a $\hat\rho$--approximate stationary point of \eqref{eq:PenProb2Introa} if 
\begin{equation}
\hat v\in\nabla g(\hat z)+\partial h(\hat z),\qquad\|\hat v\|\leq\hat{\rho} \label{eq:approx_subgrad}.
\end{equation}
Given a general quadruple $(\lam,z^-,z, v)\in\R_{++}\times\rn \times\dom h \times\R^n$,
the following simple refinement procedure shows how to obtain a pair
$(\hat z,\hat v)$ satisfying the inclusion in  \eqref {eq:approx_subgrad}
with a technically useful bound
on the residual $\hat v$ (see Proposition \ref{prop:approxsol} below).
%and to quantify its quality as
%an approximate solution of \eqref{eq:PenProb2Introa}.

%describes its is obtained in general and the
%result following it describes its quality as an approximate stationary point of
%\eqref{eq:PenProb2Introa}.

\vspace*{0.5em}

\noindent\begin{minipage}[t]{1\columnwidth}
\noindent\rule[0.5ex]{1\columnwidth}{1pt}

\noindent \textbf{Refinement procedure.}

\noindent\rule[0.5ex]{1\columnwidth}{1pt}
\end{minipage}

\vspace*{0.5em}

\noindent {\bf Input:} a scalar $M>0$, a pair of functions $(g,h)$ satisfying assumptions (A1) and (A2), and a quadruple $(\lam,z^-,z, v)\in\R_{++}\times\R^n\times\dom h \times\R^n$; 

\vspace*{1em}

\noindent {\bf Output:} a triple $(\hat z,\hat v, \Delta) \in \dom h \times \rn \times \r_{++}$ satisfying \eqref{eq:inclv'};

\begin{itemize}

\item[(0)] set % satisfying 
    \begin{equation}\label{eq:def_f}
    M_\lam:=\lambda M+1,
    %\quad \rho=\left\|w + \frac{z^--z}{\lambda} \right\|,
    \quad 
    f_\lam:=\lambda g +\frac{1}{2} \|\cdot-z^-\|^2-\langle v, \cdot \rangle, \quad h_{\lam} := \lam h;
    \end{equation}
\item[(1)]
    compute
    \begin{align}
    &\hat  z := \argmin_u \left\{ \inner{\nabla f_\lam(z)}{u-z} + \frac{M_\lambda}2 \|u-z\|^2 + h_\lam(u) \right \}, \label{eq:z_pRefProc} \\
    & \hat v:= \frac{1}\lam \left[( v+ z^--z ) +M_{\lam}(z-\hat z) \right] +  \nabla g(\hat z)-\nabla g(z), \label{eq:ref_vp} \\
   & \Delta:= (f_\lam+h_\lam)(z) - (f_\lam+h_\lam)(\hat z); \label{eq:ref_var}
    \end{align}
\item[(2)] return the triple $(\hat z,\hat v,\Delta)$.
\end{itemize}
\rule[0.5ex]{1\columnwidth}{1pt}

For the sake of brevity, we write
$(\hat z,\hat v,\Delta)=RP(\lam,z^-,z, v)$ to indicate that the triple $(\hat z,\hat v,\Delta)$
is the output
of the above refinement procedure with inputs $M$, $(g,h)$, and $(\lam,z^-,z, v)$. We now state an important property of this procedure, whose proof can be found in Appendix A.

\begin{proposition}\label{prop:approxsol}
    Let a pair of functions $(g,h)$ satisfying (A1)--(A3) and a quadruple
    $(\lam,z^-,z, v)\in\R_{++}\times\R^n\times\dom h \times\R^n$ be given and let
    $(\hat z,\hat v,\Delta) = RP(\lam,z^-,z,v)$.
    Then, $\Delta \geq 0$ and
    \begin{equation}\label{eq:inclv'}
    \hat v \in \nabla g(\hat z) + \partial h (\hat z),\quad \lam \|\hat v\|\leq \|v +z^--z\| +2\sqrt{2 M_\lam \Delta}
    \end{equation}
    where $M_\lambda$ is as in \eqref{eq:def_f}.
\end{proposition}

The above proposition shows that the pair $(\hat z, \hat v)$, computed as in \eqref{eq:z_pRefProc} and \eqref{eq:ref_vp}, clearly satisfies the inclusion in \eqref{eq:approx_subgrad} and that the quantity
$\lambda \|\hat v\|$ has an upper  bound expressed in terms of the two quantities: $\|v+z^- -z\|$ and $\sqrt{M_\lam \Delta}$. Given a tolerance $\hat \rho>0$, it will be shown in Proposition~\ref{prop:gipp10} below
that the GD framework stated next 
generates a sequence of iterates $\{(\lam_k,z_k,v_k)\}$
whose corresponding refined sequence $\{(\hat z_k,\hat v_k)\}$ obtained as
$(\hat z_k,\hat v_k)=RP(\lam_k,z_{k-1},z_k,v_k)$ for every $k \geq 1$ yields a
$\hat \rho$--approximate stationary point of \eqref{eq:PenProb2Introa}.

\vspace*{0.5em}

\noindent\begin{minipage}[t]{1\columnwidth}
\noindent\rule[0.5ex]{1\columnwidth}{1pt}

\noindent\textbf{GD framework.}

\noindent\rule[0.5ex]{1\columnwidth}{1pt}
\end{minipage}

\vspace*{0.5em} 

\noindent {\bf Input:} a scalar $M > 0$, a function pair $(g,h)$ satisfying assumptions (A1)--(A3), an initial point $z_0 \in \dom h$, and a scalar pair $(\theta, \tau) \in \R_{++}^2$;

\begin{itemize}
\item [(0)] set $\phi := g + h$ and $k=1$;
\item [(1)] find a triple $(\lambda_k, z_k, v_k) \in \R_{++} \times \dom h \times \R^n $
such that its corresponding refined triple
\[
(\hat z_k, \hat v_k,\Delta_k) = RP(\lam_k,z_{k-1},z_k,v_k)
\]
satisfies
\begin{gather}
 \|v_k+ z_{k-1} - z_k \|^2 \leq \theta\lam_k  [\phi(z_{k-1}) - \phi(z_k)],
 \label{eq:bd_prox-approx} \\
2 \left( \lam_k M + 1 \right) \Delta_k \leq \tau \|v_k+z_{k-1}-z_k \|^2; \label{eq:eps_gsm_bd}
\end{gather}

\item [(2)] set $k = k+1$ and go to step 1.

\end{itemize}
\rule[0.5ex]{1\columnwidth}{1pt}

We now make three remarks about the GD framework.
%First, step 1 should be viewed as an oracle in that it does not specify how to compute the triple 
%$(\lambda_k, z_k, v_k)$. Section~\ref{sec:AIPPmet} describes a specific instance of this scheme, namely,
%the R-AIPP method, which computes
%this triple by invoking (possibly repeatedly) an ACG variant described in
%Section~\ref{sec:Nesterov's-Method}.
First, no termination criterion is added to the GD framework so as to be able to discuss convergence rate results about its generated sequence.
A discussion of how to terminate it is given after Proposition~\ref{prop:gipp10} below.
Second, step 1 should be viewed as an oracle in that it does not specify how to compute the triple
$(\lambda_k, z_k, v_k)$. Third, Corollary~\ref{cor:exact_prox} below shows that if the stepsize $\lambda_k$
is chosen so that the 
prox subproblem \eqref{eq:penPbRegIntroa} is a strongly convex composite problem,  i.e.,
$\lam_k \in (0, 1/\underline m)$ where $\underline{m}$ is as in \eqref{eq:m_lower_def}, the point $z_k$ is chosen as its unique optimal solution, and $v_k$ is set to zero, then
the triple $(\lambda_k, z_k, v_k)$ satisfies \eqref{eq:bd_prox-approx} and \eqref{eq:eps_gsm_bd} with
$\theta = 2$ and $\tau = 0$. Thus, when
$(\theta, \tau)  \in [2,\infty) \times[0,\infty)$, we conclude that:
(i) there always
exists a triple satisfying \eqref{eq:bd_prox-approx} and \eqref{eq:eps_gsm_bd}; and,
(ii)  the GD framework can be viewed as an IPP method.
%Hence, as long as $\theta \ge 2$ and $\tau \ge 0$, the existence of a triple ??? satisfying ?? is justified.
Fourth, the R-AIPP of Section~\ref{sec:AIPPmet}, being a special instance of the GD framework,  can also be viewed
as an IPP method which chooses $(\theta, \tau)$ in the open rectangle $(2,\infty) \times(0,\infty)$ and
%approximately solves \eqref{eq:penPbRegIntroa} by using
applies an ACG variant, such as the one described in Section~\ref{sec:Nesterov's-Method}, to problem \eqref{eq:penPbRegIntroa}
in order to obtain a triple $(\lambda_k,z_k,v_k)$ satisfying \eqref{eq:bd_prox-approx} and \eqref{eq:eps_gsm_bd}.

% that the desired triple in step~1 of the GD scheme can be obtained from an exact solution of the prox subproblem in \eqref{eq:penPbRegIntroa}. Hence, one could expect that a method applied to solving \eqref{eq:penPbRegIntroa} can be used to implement step~1 of the scheme.  Section~\ref{sec:AIPPmet} describes a relaxation of such a method, namely, the R-AIPP method, which computes the desired triple by invoking (possibly repeatedly) an ACG variant, described in Section~\ref{sec:Nesterov's-Method}, to approximately solve \eqref{eq:penPbRegIntroa}.

The following result shows an important property about the sequence of iterates $\{(\lam_k, \hat z_k, \hat v_k)\}$.
%, which will be necessary in discussing how instances of GD scheme  obtain an approximate solution of \eqref{eq:PenProb2Introa} as in \eqref{eq:approx_subgrad}.

\begin{proposition}\label{prop:gipp10} The sequences of stepsizes $\{\lambda_k\}$ and iterate pairs $\{(\hat z_k, \hat v_k)\}$ satisfy
\begin{equation}
\hat v_k \in \nabla g(\hat z_k) + \pt h(\hat z_k),\quad \min_{i\leq k} \|\hat v_i\|^2 \leq \theta\left(1+2\sqrt{\tau}\right)^2 \frac{[\phi(z_0) - \phi_*]}{\Lambda_k}, \label{eq:corGD_complex_b}
\end{equation}
%{\bf add the inclusion here}
for every $k \ge 1$, where $\Lambda_k := \sum_{i=1}^k \lam_i$.

\end{proposition}

\begin{proof} Let $k\geq 1$ be fixed. The inclusion in \eqref{eq:corGD_complex_b} follows from Proposition~\ref{prop:approxsol} with $(\hat z, \hat v)=(\hat z_k, \hat v_k)$ and the definitions of $\hat z_k$ and $\hat v_k$ in step~1 of the GD framework.
To show the inequality in \eqref{eq:corGD_complex_b}, first observe that \eqref{eq:bd_prox-approx} and the definition of $\phi_*$ in \eqref{eq:PenProb2Introa} implies that
\begin{align} 
& \phi(z_0) - \phi_* \geq \sum_{i=1}^k [\phi(z_{i-1})-\phi(z_i)] 
\geq \sum_{i=1}^k \frac{\|v_i + z_{i-1} - z_i\|^2}{\theta\lam_i} \geq \frac{\Lambda_k}{\theta} \min_{i\leq k} \frac{1}{\lam_i^2} \|v_i + z_{i-1} - z_i\|^2. \label{eq:GD_complex_prf_a} 
\end{align}
%and hence that \eqref{eq:corGD_complex_a} holds. 
Now, let $i\geq1$ be arbitrary. In view of step~1 of the GD framework we have 
$(\hat z_i, \hat v_i,\Delta_i) = RP(\lam_i,z_{i-1},z_i,v_i)$. Hence Proposition~\ref{prop:approxsol}
with $(\lam,z^-,z,v, \hat v) = (\lam_i, z_{i-1}, z_i, v_i, \hat v_i)$
and \eqref{eq:eps_gsm_bd} with $k=i$ imply that
\begin{equation} \label{eq:v_hat_bd}
     \|\hat{v}_{i}\| \leq\left(1+2\sqrt{\tau}\right)\frac{\|v_{i}+z_{i-1}-z_{i}\|}{\lam_{i}}.
\end{equation}
The inequality in \eqref{eq:corGD_complex_b} now follows by combining \eqref{eq:GD_complex_prf_a} and \eqref{eq:v_hat_bd}.
\end{proof}

%{\bf move this one} However, a natural way of terminating the GD scheme would be 
%when $(\hat z,\hat v)= (\hat z_k, \hat v_k)$ satisfies the inequality in \eqref{eq:approx_subgrad}.
%Since $(\hat z_k, \hat v_k)$ automatically satisfies the inclusion in \eqref{eq:approx_subgrad} in view of Proposition~\ref{prop:approxsol},
%this termination would then guarantee that $(\hat z_k, \hat v_k)$
%is a $\hat{\rho}$--approximate stationary point of \eqref{eq:PenProb2Introa}.
%The specific GD  instance discussed in Section~\ref{sec:AIPPmet}, namely, the R-AIPP method, uses this stopping criterion.

We now make three remarks about the GD framework in light
of Proposition~\ref{prop:gipp10}. 
First, if the GD framework stops when a pair $(\hat z_k,\hat v_k)$ such that $\|\hat v_k\| \leq \hat \rho$
is found, then it follows from \eqref{eq:approx_subgrad} and the inclusion in \eqref{eq:corGD_complex_b} that
$(\hat z_k, \hat v_k)$ is a $\hat \rho$--approximate stationary point of \eqref{eq:PenProb2Introa}.
%First, in view of \eqref{eq:approx_subgrad} and the inclusion in \eqref{eq:corGD_complex_b}, a natural way to terminate the scheme is to stop whenever $\|\hat v_k\| \leq \hat\rho$ and output the pair $(\hat v, \hat z)=(\hat v_k, \hat z_k)$. Clearly, if the scheme terminates in the aforementioned manner, then the output pair $(\hat z, \hat v)$ is a $\hat \rho$--approximate stationary point of \eqref{eq:PenProb2Introa}. 
Second, if the sequence of stepsizes $\{\lambda_i\}$ satisfies  $\lim_{k\to\infty} \Lambda_k = \infty$, then it follows
from the inequality in \eqref{eq:corGD_complex_b} and assumption (A3) that the GD framework indeed stops according
to the above termination criterion.
%if the stepsizes $\{\lambda_i\}$ in the scheme satisfy  $\lim_{k\to\infty} \Lambda_k = \infty$, then there is a subsequence of the residuals $\{\|v_k\| \}$ that tends to zero. Together with the fact that, for every $k \geq 1$, the pair $(\hat z, \hat v)=(\hat z_k, \hat v_k)$ satisfies the inclusion in \eqref{eq:approx_subgrad}, a natural stopping criterion for the GD scheme is to stop when $\|\hat v_k\| \leq \hat \rho$, which is exactly when $(\hat z_k, \hat v_k)$ is a $\hat \rho$--approximate stationary point. The specific GD  instance discussed in Section~\ref{sec:AIPPmet}, namely, the R-AIPP method, uses this stopping criterion. 
Third, \eqref{eq:corGD_complex_b} indicates that
the larger the stepsizes $\lam_k$ are, the faster the quantity $\min_{i\leq k} \|\hat v_i\|$ approaches zero. 

For the remainder of this section, our goal is to show that the GD framework can be seen as a relaxation of the GIPP framework studied in \cite{WJRproxmet1}. The proof of this fact is not essential in establishing any results pertaining to the R-AIPP method in Section~\ref{sec:AIPPmet} or the R-QP-AIPP method in Section~\ref{sec:penalty} and may skipped  without any loss of continuity.

Recall that, for a given $z_0\in \dom h$ and $\sigma\in [0,1)$, the GIPP framework in \cite{WJRproxmet1} considers a sequence  $\{(\lambda_k, z_k, v_k, \varepsilon_k)\} \subseteq \R_{++} \times \dom\phi \times \R^n\times\R_+$  satisfying
\begin{gather}
{v}_k\in\partial_{{\varepsilon}_k}\left(\lambda_k\phi+\frac{1}{2}\|\cdot-z_{k-1}\|^{2}\right)(z_k), \quad 
\|{v}_k\|^{2}+2{\varepsilon}_k\leq\sigma\|{v}_k+z_{k-1}-z_k\|^{2}, \label{eq:gipp_main}
\end{gather} 
for every $k \ge 1$.
We now state a simple technical result which will not only be used in this section but also later in the analysis of the R-ACG algorithm
(see Section~\ref{sec:Nesterov's-Method}).

%---------------------------
%
%\[
%\frac1{\lambda} v \in \partial \left ( \lambda \phi + \frac12 \|\cdot - z^-\|^2 \right) (z) =
%\lambda ( \nabla g(z) + \partial h(z) ) + z-z^-
%\]
%or equivalently,
%\[
%\frac1{\lambda} v + z^--z \in \nabla g(z) + \partial h(z) 
%\]
%Observe that
%\[
%\frac1{\lambda} v + z^--z = 0
%\]
%iff $z$ solves
%\[
%z-z^- \in \partial \left ( \lambda \phi + \frac12 \|\cdot - z^-\|^2 \right) (z) 
%\]
%
%
%
%--------------------------

    \begin{lemma}\label{lem:Del_bd}
    Assume that $\varepsilon\geq0$ and $(\lam,z^-,z,v)\in\R_{++}\times\R^n\times\dom h \times \R^n$ satisfy 
    \begin{gather}
    v\in\partial_\varepsilon\left(\lam\phi+\frac{1}{2}\|\cdot-z^-\|^2\right)(z). \label{eq:Del_bd_incl}
    \end{gather}
    Then, the quantity $\Delta$ defined in \eqref{eq:ref_var}  satisfies $\Delta\leq\varepsilon$.
\end{lemma}
\begin{proof}
    Let $(\hat z,\Delta)$ be computed as in \eqref{eq:z_pRefProc} and \eqref{eq:ref_var}. It follows from \eqref{eq:epsubdiff} and \eqref{eq:Del_bd_incl} that 
    \[
    \lambda\phi(z')+\frac{1}{2}\|z'-z^-\|^{2} \geq \lambda\phi(z)+\frac{1}{2}\|z-z^-\|^{2} + \langle v, z' - z \rangle - \varepsilon \quad \forall z' \in \R^n.
    \]
    Considering the above inequality at the point $z'=\hat z$, along with some algebraic manipulation, we have
    \begin{align*}
    \varepsilon & \geq \left[\lam\phi(z)+\frac{1}{2}\|z-z^-\|^{2} - \langle v, z \rangle \right] - \left[\lam\phi(\hat z)+\frac{1}{2}\|\hat z-z^-\|^{2} - \langle v, \hat z \rangle \right] = \Delta
    \end{align*}
    where the last equality is due to the definitions of $\phi$ and $\Delta$ given in \eqref{eq:PenProb2Introa} and \eqref{eq:ref_var}, respectively. 
\end{proof}

The following result shows the relationship between the GIPP framework of
\cite{WJRproxmet1} and the GD framework of this section.

\begin{proposition}\label{prop:GIPP1}
    If, for some $z_{k-1} \in \dom h$, constant $\sigma \in [0,1)$, and index $k \ge 1$, the quadruple
    $(\lam_k,z_k,v_k,\varepsilon_k)$ satisfies
    \eqref{eq:gipp_main}, then $(\lam_k,z_k, v_k)$ satisfies \eqref{eq:bd_prox-approx} and \eqref{eq:eps_gsm_bd}
    for any $\theta \geq  2/(1-\sigma)$ and $\tau \geq \sigma (\lam_k M + 1)$.
    As a consequence, if $\sup \{ \lam_k : k\geq 1 \} < \infty $, then
     every instance of the GIPP framework is an instance of the GD framework for any $(\theta,\tau)$ satisfying
     \begin{equation}\label{eq:gipp_gd_cond}
     \theta \geq \frac{2}{1-\sigma}, \quad \tau \geq \sup \left\{ \sigma (\lam_k M + 1) : k \geq 1\right\}.
     \end{equation}
\end{proposition}

%\begin{proposition}\label{prop:GIPP1}
%    Let $\{(z_k,v_k,\varepsilon_k,\lambda_k)\}$ be a sequence satisfying 
%    \eqref{eq:gipp_main} for some $\sigma \in (0,1)$, and suppose that $\bar \lam := \sup \{ \lam_k : k\geq 1 \} < \infty $. Then, for every $k\geq 1$, the corresponding refined triple $(\hat z_k, \hat v_k,\Delta_k) = RP(\lam_k,z_{k-1},z_k, v_k)$ satisfies \eqref{eq:bd_prox-approx} and \eqref{eq:eps_gsm_bd} with $\theta = 2/(1-\sigma)$ and $\tau = \sigma (\bar\lam M + 1)$. 
%    As a consequence, every instance of the GIPP framework is an instance of the GD scheme in which $\theta = 2/(1-\sigma)$ and $\tau = \sigma (\bar \lam M + 1)$.
%\end{proposition}
\begin{proof}
    The proof that $(\lambda_k, z_k, v_k)$ satisfies
    \eqref{eq:bd_prox-approx} with $\theta=2/(1-\sigma)$ can be found in \cite[Proposition 5(a)]{WJRproxmet1}.
    Now, let $k\geq1$ and observe that from Lemma~\ref{lem:Del_bd} with $(\lam,z^-,z,v) = (\lam_k, z_{k-1}, z_k, v_k)$ and $\varepsilon=\varepsilon_k$ we have $\Delta \leq \varepsilon_k$. It follows from the last inequality  and  the inequality in \eqref{eq:gipp_main} that  $2 \Delta \leq \sigma\|v_k + z_{k-1} - z_k\|^2$. Combining the previous inequality with the assumption on $\tau$ now shows that $(\lambda_k, z_k, v_k)$ satisfies \eqref{eq:eps_gsm_bd}. The second part of the proposition follows immediately from the first part and condition \eqref{eq:gipp_gd_cond}.
\end{proof}

The above proposition shows that if $\{\lam_k\}$ is bounded and the parameter triple $(\sigma,\theta,\tau)$ satisfies \eqref{eq:gipp_gd_cond}, then the condition for finding an iterate $(\lam_k,z_k,v_k)$ in the GD framework is more relaxed than the condition for finding an iterate $(\lam_k,z_k,v_k,\varepsilon_k)$ in the GIPP framework. As a consequence,  under
the conditions in \eqref{eq:gipp_gd_cond}, an optimization algorithm (such as the R-ACG algorithm of Section~\ref{sec:Nesterov's-Method}) applied to \eqref{eq:penPbRegIntroa}
is expected to find the triple
$(\lam_k,z_k,v_k)$ for the GD framework faster than the quadruple $(\lam_k,z_k,v_k,\varepsilon_k)$ for the GIPP framework.

The following corollary justifies the third remark following the GD framework.
%??? We now show if $\lambda_k\in(0,1/m)$, the point $z_k$ is chosen as the optimal solution of \eqref{eq:penPbRegIntroa}, and $v_k$ is set to zero, then the triple $(\lambda_k, z_k, v_k)$ satisfies the condition in step~1 of the scheme with $\theta = 2$ and $\tau  = 0$.

\begin{corollary} \label{cor:exact_prox}
    Let $z_{k-1} \in \dom h$ and $\lam_k \in (0,1/ \underline{m})$ be given, where $\underline{m}$ is as in \eqref{eq:m_lower_def}. Then, \eqref{eq:penPbRegIntroa}
     has a unique global minimum $z_k$ and the triple $(\lam_k, z_k, v_k) \in \r_{++} \times \dom h \times \rn$ where
     $v_k=0$ satisfies \eqref{eq:bd_prox-approx} and \eqref{eq:eps_gsm_bd} with $\theta = 2$ and $\tau  = 0$. 
\end{corollary}

\begin{proof}
The existence and unique uniqueness of $z_k$ follows from the fact that  $\phi + \|\cdot-z_{k-1}\|^2/\lam_k$ is strongly convex. Moreover, the fact that $z_k$ is the unique global minimum of \eqref{eq:penPbRegIntroa} implies that the quadruple
    $(\lam_k,z_k,v_k,\varepsilon_k)$, where $(v_k,\varepsilon_k)=(0,0)$, satisfies \eqref{eq:gipp_main} with $\sigma=0$.
The conclusion of the corollary now follows immediately from the first part of Proposition~\ref{prop:GIPP1} with $\sigma=0$.
%
% and the fact that
%the quintuple $\{(z_k,v_k,\varepsilon_k,\lambda_k)\}$
%????? Since is $z_k$ is a solution of \eqref{eq:penPbRegIntroa} and $\phi=g+h$, we have that $z_k$ satisfies \eqref{eq:gipp_main} with $\sigma = 0$ and $(v_k, \varepsilon_k) = (0,0)$. It now follows from Proposition~\ref{prop:GIPP1} that the refined triple $(\hat z_k, \hat v_k,\Delta_k)= RP(\lam_k,z_{k-1},z_k, v_k)$ satisfies \eqref{eq:bd_prox-approx} and \eqref{eq:eps_gsm_bd} with $\theta = 2$ and $\tau  = 0$. Hence, the triple $(\lam_k, z_k, v_k)$ satisfies \eqref{eq:bd_prox-approx} and \eqref{eq:eps_gsm_bd}.
\end{proof}

%%%%%%%%%%%%%%%%%%%%%%%%%%%%%%%%%%%%%%%%
%%%%%%%%%%%%%%%%%%%%%%%%%%%%%%%%%%%%%%%%

\section{A relaxed accelerated composite gradient algorithm}\label{sec:Nesterov's-Method}

This section presents and analyzes an ACG variant, namely, the R-ACG algorithm, which is used as an important tool in the development of the R-AIPP method of
Section~\ref{sec:AIPPmet}.
More specifically, the R-AIPP method can be viewed as a special instance of the GD framework where step 1 is implemented by repeatedly calling the ACG variant of this section.

Before describing the variant, we consider its assumptions as well as the problem that it solves.
First, we describe the assumptions. Let $\widetilde \phi : \r^n \to (-\infty, \infty]$ be given and assume that it
can be decomposed as
$\widetilde \phi=\widetilde \phi^{(s)}+\widetilde \phi^{(n)}$ where:
\begin{itemize}[align=left]
    \item [(B1)]$\widetilde \phi^{(n)}\in \bConv$;
    \item [(B2)]$\widetilde \phi^{(s)}$ is  a differentiable function on $\dom \widetilde\phi^{(n)}$ such that for some $\widetilde{M} > 0$, 
    $$
    \widetilde\phi^{(s)}(u) \leq   \ell_{\widetilde\phi^{(s)}}(u;x) + \frac{\widetilde{M}}{2}\|u-x\|^2\quad \forall u,x \in \dom \psi^{(n)}.
    $$
\end{itemize}

%
%
%
%The ACG variant is a modified version of a particular ACG variant for solving the
%minimization problem
%, as mentioned before, is a special instance of the GD scheme. In particular, when this variant is applied to the optimization problem 
%\begin{gather}
%\min_x\left\{\widetilde\phi(x)+\frac{1}{2}\|x-x_0\|^2\right\} \label{eq:acg_optim}
%\end{gather}
%for some $x_0\in\dom\widetilde\phi$, it will be shown that under certain conditions it is able to procure a solution that implements step 1 of the GD scheme. 

%Before proceeding, let us assume that $\widetilde\phi$ in the aforementioned optimization problem can be decomposed as
%$\widetilde \phi=\widetilde \phi^{(s)}+\widetilde \phi^{(n)}$ where:
%\begin{itemize}
%    \item [(B1)]$\widetilde \phi^{(n)}:\r^n\rightarrow (-\infty,+\infty]$ is a proper closed  convex  function;
%    \item [(B2)]$\widetilde \phi^{(s)}$ is  a differentiable function on $\dom \widetilde\phi^{(n)}$ such that for some $\widetilde{M}>0$, 
%    $$
%    \widetilde\phi^{(s)}(u)-[ \widetilde\phi^{(s)}(x)+\langle \nabla\widetilde\phi^{(s)}(x),u-x\rangle]\leq  \frac{\widetilde{M}}{2}\|u-x\|^2\quad \forall x, u \in \dom \psi^{(n)}.
%    $$
%\end{itemize}

We now describe our problem of interest in this section.
% problem ?????

%\noindent We now state a problem of interest for the ACG variant that is an extension of solving \eqref{eq:acg_optim} and is in the context of implementing part of the GD scheme. 

%ACG algorithms for solving the convex optimization problem
%$\min \{ \psi(z) : z \in \R^n \}$ where $\psi \in \bConv$ has a certain convex composite structure have been studied in several papers (see for example \cite{beck2009fast,MontSvaiter_fista}). These algorithms, and their variants, can be leveraged to solve the following problem, which will be shown to be closely related to the GD scheme. 

\vgap

\noindent{\bf Problem A:} Given $\widetilde \phi:  \R^n \to (-\infty,+\infty]$ satisfying the above assumptions,
a point $x_0 \in \R^n$, and a pair of parameters $(\theta, \tau)\in(2,\infty)\times (0, \infty)$,
the problem is to find a triple $(x,u,\eta) \in \R^n \times \R^n \times \R_+$ such that
\begin{gather}
    \|x_{0}-x+u\|^{2}\leq \theta \left[\widetilde \phi (x_{0})- \widetilde \phi (x)\right], \label{eq:ACG_descent}\\
    u \in \partial_{\eta} \left( \widetilde \phi + \frac12 \|\cdot-x_0\|^2 \right) (x), \quad
    2 \left(\widetilde M + 1\right) \eta \leq \tau \|x_{0}-x+u\|^{2}. \label{eq:ACG_incl_eta}
\end{gather}

%For the sake of future reference, we give a formal description of a composite structure of $\psi$.
%
%\begin{definition}\label{def:comp_struc}
%    For scalars $\mu \ge 0$ and $L \ge 0$, %and functions $\psi^{(s)}$ and $\psi^{(n)}$ ????????? 
%    the quadruple $(\psi^{(s)},\psi^{(n)};\mu,L)$ is called a composite structure of  the (possibly nonconvex)  function $\psi$ if
%    $\psi=\psi^{(s)}+\psi^{(n)}$ and the following conditions hold:
%    \begin{itemize}
%        \item [(B1)]$\psi^{(n)}:\r^n\rightarrow (-\infty,+\infty]$ is a proper, closed and $\mu$-strongly convex  function;
%        \item [(B2)]$\psi^{(s)}$ is  a differentiable function on $\dom \psi^{(n)}$ such that
%        $$
%        \psi^{(s)}(u)-[ \psi^{(s)}(x)+\langle \nabla\psi^{(s)}(x),u-x\rangle]\leq  \frac{L}{2}\|u-x\|^2,\quad \forall x, u \in \dom \psi^{(n)}.
%        $$
%    \end{itemize}
%    When $\psi^{(s)}$ is convex, we refer to $(\psi^{(s)},\psi^{(n)};\mu,L)$ as a convex composite structure of $\psi$.
%\end{definition}

%\noindent To show the relationship between the above problem to the scheme in Section~\ref{sec:background}, it is first assumed that 
   % When $\psi^{(s)}$ is convex, we refer to $(\psi^{(s)},\psi^{(n)};\mu,L)$ as a convex composite structure of $\psi$.

%Clearly, a necessary condition for $\psi$ to have a $\mu$-convex composite structure is that it be $\mu$-strongly convex.
%where $\widetilde\phi:=\lam \phi$ ?????.
%when $\psi^{(s)}$ is convex. When $\psi^{(s)}$ is nonconvex, the inclusion in \eqref{eq:ACG_incl_eta} is removed as a requirement. 

\vgap

The following simple result shows how the ability to solve Problem A
allows us to implement the ``step 1'' oracle in the GD framework.

\begin{proposition} \label{prop:ACG_gd_link}
Assume that $\phi=g+h$ satisfies conditions (A1) and (A2), and
 let $z_{k-1}\in \dom h$ be given.
Then the following statements hold:
\begin{itemize}
     \item[(a)] if $(x,u)$ satisfies \eqref{eq:ACG_descent} with $(\widetilde \phi, \widetilde M, x_0)=(\lambda \phi, \lam M,z_{k-1})$ for some
     $\lam>0$,
then  the triple $(\lam_k,z_k,v_k):=(\lam,x,u)$ satisfies \eqref{eq:bd_prox-approx}; 
     \item[(b)]
      if $(x,u,\eta)$ solves Problem A with input $(\widetilde \phi,\widetilde M, x_0)=(\lambda \phi, \lam M, z_{k-1})$
      for some $\lam>0$, then the triple $(\lam_k,z_k,v_k)=(\lam,x,u)$
solves step~1 of the GD framework.
%     if $(x,u,\eta)$ satisfies \eqref{eq:ACG_incl_eta} with $(\widetilde \phi,x_0)=(\lam \phi, z_{k-1})$ for some
%     $\lam>0$,
%then the pair $(z_k,v_k):=(x,u)$ satisfies  \eqref{eq:eps_gsm_bd}. 
\end{itemize}
%As a consequence, if $(x,u,\eta)$ solves Problem A with input $(x_0,\widetilde \phi):=(z_{k-1},\lam\phi)$ for some $\lam>0$, then $(\lam_k,z_k,v_k)=(\lam,x,u)$
%solves step 1 of the GD scheme. 
\end{proposition}

\begin{proof} 
(a) Assume that $(x,u)$ satisfies \eqref{eq:ACG_descent}. It follows from the fact that $(\lam, x, u) = (\lam_k, z_k, v_k)$ and the definition of $\widetilde\phi$ that
\begin{align*}
\left\Vert  z_{k-1} - z_k+v_k \right\Vert^2  \leq \theta\left[\widetilde\phi(z_{k-1}) - \widetilde\phi(z_k) \right] = \theta\lam_k \left[\phi(z_{k-1}) - \phi(z_k)\right] 
\end{align*}
and thus the triple $(\lambda_k, z_k, v_k)$ satisfies \eqref{eq:bd_prox-approx}.

(b)  Assume that $(x,u,\eta)$ satisfies \eqref{eq:ACG_incl_eta}  and define $\varepsilon:=\eta$ and $(z^-, z, v) := (x_0, x, u)$. Moreover, let $\Delta$ be computed as in \eqref{eq:ref_var} with $\hat z$ as in \eqref{eq:z_pRefProc}. It follows from Lemma~\ref{lem:Del_bd}, the definition of $\widetilde\phi$, the fact that $\eta=\varepsilon$, and the inclusion in \eqref{eq:ACG_incl_eta} that $\Delta\leq\eta$. Using the inequality in \eqref{eq:ACG_incl_eta} and the fact that $(x_0,x,u)=(z_{k-1},z_k,v_k)$ gives $2 (\widetilde M + 1)\Delta\leq\tau\|z_{k-1}-z_k+v_k\|^2$ and thus the pair $(z_k,v_k)$ satisfies \eqref{eq:eps_gsm_bd} in view of the definition of $\widetilde M$. As a consequence, the triple $(\lam_k,z_k,v_k)$ solves step~1 of the GD framework.
\end{proof}

The R-ACG algorithm presented below,
which is a modified ACG variant for minimizing the function $\psi := \widetilde \phi + \|\cdot-x_0\|^2/2$,
solves  Problem A under the assumption that  $\psi$ is convex (see Proposition~\ref{prop:nest_complex}(c) below).
As a consequence, it can be used to implement step~1 of the GD framework whenever $\lambda_k$ is sufficiently small. More specifically,
since $\lam_k \phi + \|\cdot - z_{k-1}\|^2/2$ is clearly convex whenever $\lam_k$ is chosen in $(0,1/\underline{m}]$, where $\underline{m}$ is as in \eqref{eq:m_lower_def},
we can use the R-ACG algorithm to solve problem A with $\widetilde \phi=\lam_k \phi$ and $x_0=z_{k-1}$, and hence
the ``step 1'' oracle in the GD framework in view of Proposition~\ref{prop:ACG_gd_link}(b).
In fact, the AIPP method developed in \cite{WJRproxmet1} is an instance of the GIPP framework (and hence an instance of the GD framework) in which, given an upper bound $m$ on $\underline{m}$, it chooses $\lambda_k=1/(2m)$ for all $k$ and in which step 1 is implemented with a single call to the R-ACG algorithm presented below.

However, our main goal in this paper is the development of an instance of the GD framework which aggressively chooses $\lam_k$ (possibly) much larger than $1/ \underline{m}$ since, according to Proposition~\ref{prop:gipp10}, this strategy can potentially reduce its number of iterations.
In this regard, the R-ACG algorithm presented below accepts as input a function $\widetilde\phi$
of the form $\widetilde \phi=\lam \phi$ for some $\lam>0$ in which $\widetilde\phi + \|\cdot-x_0\|^2/2$ is not necessarily convex,
and terminates with either failure or by finding a triple $(x,u,\eta)$ satisfying \eqref{eq:ACG_descent} within ${\cal O}(\widetilde{M}^{1/2} \log^+_1 \widetilde M )$ iterations (see statements (a) and (b) of Proposition~\ref{prop:nest_complex}  below).
Clearly, in the second case,
the triple $(\lam_k,z_k,v_k)=(\lam,x,u)$ is guaranteed to  satisfy \eqref{eq:bd_prox-approx} but not necessarily \eqref{eq:eps_gsm_bd} (see Proposition~\ref{prop:ACG_gd_link}(a)).
If \eqref{eq:eps_gsm_bd} is satisfied then the R-ACG algorithm clearly provides a solution of the ``step 1'' oracle of the GD framework; otherwise, the stepsize $\lam$ is considered large.
The R-AIPP method of Section~\ref{sec:AIPPmet} is an instance of the GD framework which attempts to provide a solution of its ``step 1'' oracle in this manner and adaptively reduces $\lam$ whenever it is found to be  large.

\noindent\begin{minipage}[t]{1\columnwidth}
\noindent\rule[0.5ex]{1\columnwidth}{1pt}

\noindent\textbf{R-ACG algorithm.}

\noindent\rule[0.5ex]{1\columnwidth}{1pt}
\end{minipage}

\vspace*{0.5em}

\noindent {\bf Input:} a scalar  $\widetilde{M}>0$, a function pair $(\widetilde\phi^{(s)},\widetilde\phi^{(n)})$ satisfying assumptions (B1) and (B2), an initial point $x_0 \in \dom \widetilde\phi^{(n)}$, and a pair of parameters  $(\theta,\tau) \in   (2,\infty)\times (0, \infty)$;

\vspace*{1em}

\noindent {\bf Output:} a triple $(x, u, \eta) \in \dom \widetilde\phi^{(n)} \times \rn \times \r_+$ satisfying \eqref{eq:ACG_descent} or a \textbf{failure} status;

\begin{itemize}
\item[(0)] set $y_{0}=x_{0}$, $A_{0}=0$, $\Gamma_0\equiv0$, $j=1$, and define
\begin{gather} \label{eq:acg_defs}
\psi^{(s)} := \widetilde\phi^{(s)} + \frac{1}{4}\|\cdot-x_0\|^2,\qquad \psi^{(n)} := \widetilde\phi^{(n)}  + \frac{1}{4}\|\cdot-x_0\|^2, \qquad 
\psi := \psi^{(s)} + \psi^{(n)}, \\
\widetilde{L} := \widetilde{M} + \frac{1}{2}, \qquad \mu := \frac{1}{2};
\end{gather}
\item[(1)]  compute
\begin{align}
 A_{j}  &=A_{j-1}+\frac{\mu A_{j-1} + 1+\sqrt{(\mu A_{j-1} + 1)^2+4\widetilde{L}(\mu A_{j-1}+ 1)A_{j-1}}}{2\widetilde{L}},\label{Formula_Aj}\\
\widetilde{x}_{j-1}  &=\frac{A_{j-1}}{A_{j}}x_{j-1}+\frac{A_{j}-A_{j-1}}{A_{j}}y_{j-1},\quad\Gamma_{j}=\frac{A_{j-1}}{A_{j}}\Gamma_{j-1}+\frac{A_{j}-A_{j-1}}{A_{j}}\ell_{\psi^{(s)}}(\cdot,\widetilde x_{j-1}),\label{Formula_Gammaj}\\
y_{j} &=\argmin_{y} \left\{\Gamma_{j}(y)+\psi^{(n)}(y)+\frac{1}{2A_{j}}\|y-y_{0}\|^{2}\right\},\label{def:yj}\\\
x_{j} & =\frac{A_{j-1}}{A_{j}}x_{j-1}+\frac{A_{j}-A_{j-1}}{A_{j}}y_{j}\label{def:xj}
\end{align}
and set
\begin{equation}\label{def:ujetaj}
u_{j}=\frac{y_0-y_{j}}{A_{j}}, \quad
\eta_{j}= \max \left\{ \psi(x_{j})-\Gamma_{j}(y_{j})- \psi^{(n)}(y_{j})-\langle u_{j},x_{j}-y_{j}\rangle, 0 \right\};
\end{equation}
\item[(2)]
if  both inequalities 
\begin{align}
&\|A_{j}u_{j}+x_{j}-x_{0}\|^{2}+2A_{j}\eta_{j} \le  \|x_{j}-x_{0}\|^2, \label{mainIneqAACG}\\[2mm]
&\psi(x_{0}) \geq \psi(x_{j}) +\langle u_{j},x_{0}-x_{j}\rangle -\eta_{j},\label{crit:subIneqACG}
\end{align}
hold, then  go to step 3; otherwise,  stop with {\bf failure};
\item[(3)]   if both inequalities
\begin{gather}
2 \left(\widetilde M + 1\right) \eta_{j} \leq \tau \|x_{0}-x_{j}+u_{j}\|^{2}\label{ineq:ACGHPE_algo}, \\
 \|x_{0}-x_j+u_j\|^{2}\leq \theta\left[\widetilde\phi(x_{0})-\widetilde\phi(x_j)\right], \label{crit:ACGdescent_algo}
\end{gather}
hold, then return $(x,u,\eta)=(x_j, u_j,\eta_j)$; otherwise,
increment $j = j+1$ and go to step 1.

\end{itemize}
\noindent\rule[0.5ex]{1\columnwidth}{1pt}

%---------------------------------
%\textbf{Remarks about the Relaxed ACG}
%---------------------------------
%\begin{itemize}
%    \item Compared to the old ACG, this variant requires the additional variables $\xi, 
%, \theta_2$, and $\lambda$; it loses the variable $\sigma$.
%    \item The inequalities \eqref{mainIneqAACG}, \eqref{crit:subIneqACG}, along with the fact that $A_j\to\infty$ imply \eqref{ineq:ACGHPE_algo} and \eqref{crit:ACGdescent_algo}. The proof uses the fact that \eqref{eq:gipp2} can be shown as an intermediate step, and the rest follows from various algebraic manipulations. 
%    \item If $\psi$ is convex, then \eqref{ineq:ACGHPE_algo} implies \eqref{ineq:ACGHPE} with $(x,x_0,u) = (x_j, x_0, u_j)$. The rough proof is that in the convex case, \eqref{eq:gipp0} holds everywhere, and so \eqref{eq:DelEps} holds implying that $\Delta \leq \eta_j/\lambda$. The inequality \eqref{ineq:ACGHPE_algo} completes the proof.
%    \item The condition $\theta_1 > 2$ is stronger than the GD framework where $\theta_1 > 0$.
%\end{itemize}
%------------------------------------------------------------------------------------------------------------------------------

Some comments about the above algorithm are in order. First, step 1 is essentially a standard step of an ACG variant
(see, for example, \cite{YHe2,WJRproxmet1}) applied to the problem $\min\{\widetilde \phi(x) + \|x-x_0\|^2/2 : x\in \rn\}$ with the exception that it also computes in \eqref{def:ujetaj}
the quantities $u_j$ and $\eta_j$ which, together with
$x_j$, determine the termination criteria for the method.
Second, it is shown in \cite[Lemma 9]{WJRproxmet1}  that a simplified version of the above algorithm, namely, one that does not include the
two tests performed in step 2 and stops whenever the inequality in \eqref{eq:gipp_main} is satisfied with 
$(z_{k-1}, z_k, {v}_k, {\varepsilon}_k)=(x_0, x_j, u_j, \eta_j)$, implements step 1 of the GIPP framework in \cite{WJRproxmet1}. 
%Fourth, in contrast to \eqref{eq:gipp_subgrad}, which is impossible to be directly verified since it consists of an infinite number of inequalities, inequality~\eqref{crit:ACGdescent_algo} can be easily verified. 
%Fifth, when $\psi^{(s)}$ is convex, both \eqref{mainIneqAACG} and \eqref{crit:subIneqACG} are always satisfied
%(see Proposition~\ref{prop:nest_complex}(c) below), and hence the method will terminate successfully.
Finally,  it is well-known (see, for example, \cite[Proposition 2.3]{YHe2}) that the scalar $A_j$ updated according to \eqref{Formula_Aj} satisfies
\begin{equation}\label{increasingAj}
%A_{j}\geq\frac{1}{L}\left(1+\sqrt{\frac{\mu}{4L}}\right)^{2(j-1)}\quad \forall  j\geq 1.
A_{j}\geq\frac{1}{L}\max\left\{\frac{j^{2}}{4},\left(1+\sqrt{\frac{\mu}{4\widetilde{L}}}\right)^{2(j-1)}\right\}\quad \forall  j\geq 1.
\end{equation}

The next result establishes the iteration-complexity bound and some properties of the R-ACG algorithm. 

%-------------------------------- Remove this part ------------------------
%
%\begin{proposition}\label{prop:nest_complex} ????? 
%The relaxed ACG algorithm stops (either with success or failure) in at most
%\[
%\]
%iterations. Moreover, if $\psi^{(s)}$ is convex, then the following statements hold about the relaxed ACG algorithm:
%\begin{itemize}
%\item[(a)] its generated sequence $\{(A_j,x_j,u_j,\eta_j)\}$ satisfies
%\begin{equation} \label{ineq:NestHPEconvex}
%\eta_j=\hat{\eta}_j\geq 0, \quad u_j\in  \partial_{\eta_{j}}\psi(x_j) ,\quad\|A_{j}u_{j}+x_{j}-x_{0}\|^{2}+2A_{j}\eta_{j}\le\|x_{j}-x_{0}\|^{2}\quad \forall j\geq 1;
%\end{equation}
%\item[(b)] it always terminates with success and hence its output ??? satisfies 
%\item[(c)] 
%\end{itemize}
%\end{proposition}
%
%----------------------------------------------

\begin{proposition}\label{prop:nest_complex} 
The R-ACG algorithm satisfies the following statements:
\begin{itemize}
\item[(a)]
it stops (either with success or failure) in at most 
\begin{equation}\label{eq:numbiter_ACG}
\left\lceil 1+\left(\sqrt{2\widetilde{M}+1}\right)\log_{1}^{+}\left(C\left[2\widetilde{M}+1\right]\right)\right\rceil
\end{equation}
iterations, where
\begin{equation} \label{eq:C_def}
C := \max\left\{\left[1 + \sqrt{\frac{\widetilde M + 1}{\tau}}\right]^2, \left[1+\sqrt{\frac{\theta}{\theta-2}}\right]^2 \right\};
\end{equation}
\item[(b)] 
if it stops with success then its output $(x,u,\eta)$ satisfies
\begin{gather} \label{eq:ACG_term_descent}
 \|x_{0}-x+u\|^{2}\leq \theta\left[\widetilde\phi(x_{0})-\widetilde\phi(x)\right];
\end{gather}
%????? given a scalar $\lambda > 0$, if $\psi = \lambda \phi + \|\cdot - x_0\|^2/2$ and the method stops successfully then its output $(x,u)$ together with $(x_0, \lambda)$ satisfy \eqref{eq:bd_prox-approx} with $(\lambda_k, z_{k-1}, z_k, w_k) = (\lambda, x_0, x, u/\lambda)$;
\item[(c)]
if $\widetilde\phi^{(s)}+\|\cdot-x_0\|^2/2$ is convex
then it always terminates with success and its output $(x,u,\eta)$ solves Problem~A.
%satisfies,
%in addition to \eqref{??}, the followwing conditions
%%, in which case there exists $\eta  \ge 0$ such that
%\begin{gather}
%u \in \partial_{\eta} \left(\widetilde\phi+\frac{1}{2}\|\cdot-x_0\|^2\right) (x), 
%\quad 2 \eta \leq\sigma \|x_{0}-x+u\|^{2}. \label{eq:ACG_strong_incl}
%\end{gather}
%if in addition  $\tau > \sigma$, then
%\[
%\widetilde \phi(x) - \widetilde \phi(x') \le \frac{\tau}{2(\tau-\sigma)} \|x_0-x'\|^2 \qquad\forall x' \in \R^n.
%\]

%and hence both \eqref{eq:bd_prox-approx} and \eqref{eq:eps_gsm_bd} are satisfied with $(\lambda_k, z_{k-1}, z_k, w_k) = (\lambda, x_0, x, u/\lambda)$
\end{itemize}
\end{proposition}

\begin{proof}
(a) See Appendix~\ref{app:r_acg}.

(b) This follows from the fact that when the R-ACG algorithm stops with success, the last iterate $(x,u)=(x_j, u_j)$ satisfies \eqref{crit:ACGdescent_algo}.

(c) %It is well-known (see for example\cite[Proposition 8(c)]{WJRproxmet1}) that if $\psi^{(s)}$ is convex, then the iterate $(x_j, u_j, \eta_j, A_j)$ satisfies\eqref{mainIneqAACG} with $\xi=1$ and the inclusion $u_j \in \partial_{\eta_j} \psi(x_j)$ for every $j\geq 1$.Using  the fact that the latter inclusion and \eqref{eq:epsubdiff} imply \eqref{crit:subIneqACG} and noting step 2 of the R-ACG algorithm, we then conclude that it does not terminate with failure. Hence, it follows from statement (a) that it must terminate with success.It then follows from the above inclusion and the fact that the last iterate $(x,u,\eta)=(x_j, u_j, \eta_j)$ satisfies \eqref{ineq:ACGHPE_algo} that $\eta=\eta_j$ fullfils the conclusion of (c).
It follows from \cite[Proposition 8(c)]{WJRproxmet1} that if $\widetilde\phi^{(s)}+\|\cdot-x_0\|^2/2$ is convex, then the iterate $(x_j, u_j, \eta_j, A_j)$ satisfies
\eqref{mainIneqAACG} and the inclusion $u_j \in \partial_{\eta_j} (\widetilde\phi+\|\cdot-x_0\|^2/2)(x_j)$ for every $j\geq 1$.
Hence, since the aforementioned inclusion and the definition of $\psi$ in \eqref{eq:acg_defs} imply \eqref{crit:subIneqACG},  we conclude that  the R-ACG algorithm does not terminate with failure (see step~2). As a consequence, it follows from statement (a) that it must terminate with success.
It then follows from the previous inclusion, and the fact that the last iterate $(x,u,\eta):=(x_j, u_j, \eta_j)$ satisfies \eqref{ineq:ACGHPE_algo}, that
$\eta$ fulfills \eqref{eq:ACG_incl_eta}. 
%The last conclusion of (c) follows immediately by Lemma~\ref{lem:auxNewNest} in the Appendix and the fact that $0<\sigma< \tau \leq 1$ (by assumption and step~0).
\end{proof}

%%%%%%%%%%%%%%%%%%%%%%%%%%%%%%%%%%%%%%%%%%%
%%%%%%%%%%%%%%%%%%%%%%%%%%%%%%%%%%%%%%%%%%%

%%%%%%%%%%%%%%%%%%%%%%%%%%%%%%%%%%%%%%%%%%%
%%%%%%%%%%%%%%%%%%%%%%%%%%%%%%%%%%%%%%%%%%%

\section{A relaxed accelerated inexact proximal point method}\label{sec:AIPPmet}

This section states and analyzes a relaxed variant of the AIPP method proposed in \cite{WJRproxmet1}, namely, the R-AIPP method, for computing an approximate stationary point of \eqref{eq:PenProb2Introa} as in \eqref{eq:approx_subgrad}. 

The  R-AIPP method stated below is an instance of the GD framework which implements its step~1  by repeatedly invoking the ACG variant in Section~\ref{sec:Nesterov's-Method} and thereby generates the method's iteration sequence. More specifically, if $z_{k-1}$ denotes the previous iterate in the GD framework and $\lam:=\lam_k$ then the R-ACG algorithm is invoked to attempt to solve Problem A with curvature $\widetilde{M}$, function pair $(\widetilde\phi^{(s)},\widetilde\phi^{(n)})$, and initial point $x_0$ given by
\begin{equation*}
\widetilde{M} =\lam M, \quad \widetilde\phi^{(s)} = \lam g, \quad \widetilde\phi^{(n)} = \lam h, \quad x_0 = z_{k-1}.
\end{equation*}
If it succeeds, it obtains a pair $(x,u)$ which will satisfy condition \eqref{eq:ACG_descent} of Problem A. Consequently, if the triple $(\lambda_k, z_k, v_k) = (\lam, x, u)$ satisfies \eqref{eq:eps_gsm_bd}, then it is a solution of step 1 of the GD framework. If the R-ACG algorithm declares failure or the triple does not satisfy \eqref{eq:eps_gsm_bd}, then the stepsize $\lambda$ is reduced and the above procedure is repeated.

\noindent\begin{minipage}[t]{1\columnwidth}
\noindent\rule[0.5ex]{1\columnwidth}{1pt}

\noindent\textbf{R-AIPP method.}

\noindent\rule[0.5ex]{1\columnwidth}{1pt}
\end{minipage}

\vspace*{0.5em}

\noindent {\bf Input:} a tolerance $\hat\rho>0$, a scalar $M>0$, a function pair $(g, h)$ satisfying assumptions (A1)--(A3), an initial point $z_0 \in \dom h$, a scalar $\lam_0 > 0$, and a pair of parameters  $(\theta,\tau) \in   (2,\infty)\times (0, \infty)$;

\vspace*{1em}

\noindent {\bf Output:} a pair $(\hat z, \hat v) \in \dom h \times \rn$ satisfying \eqref{eq:approx_subgrad};

\begin{itemize}
\item [(0)] set $\lam = \lam_0$ and $k=1$;

\item [(1)] apply the R-ACG algorithm to \textbf{Problem A} in Section~\ref{sec:Nesterov's-Method} with inputs $\widetilde M$, 
$(\widetilde\phi^{(s)},\widetilde\phi^{(n)})$, $x_0$, and $(\theta,\tau)$, where
\begin{gather}
\widetilde{M} := \lam M, \quad \widetilde\phi^{(s)} := \lam g, \quad \widetilde\phi^{(n)} := \lam h, \quad x_0 := z_{k-1}; \label{eq:AIPPpsi}
\end{gather}
if the R-ACG algorithm stops with {\bf failure}  then set $\lambda=\lambda/2$ and repeat this step;
otherwise, let $(x,u,\eta)$ denote its output triple and go to step 2;

\item[(2)] compute  $(\hat z,\hat v, \Delta)=RP(\lam,z_{k-1},x,u)$ through the refinement procedure;
 if $$2 (\lam M + 1) \Delta > \tau \|u + z_{k-1} - x \|^2,$$ then set $\lambda=\lambda/2$ and go to step 1;
otherwise, set
$$(\lambda_k,z_k,v_k)=(\lam,x,u), \quad (\hat z_k,\hat v_k )= (\hat z,\hat v),$$ 
and go to step 3;
\item[(3)]
if $\hat v_k$ satisfies
\begin{equation}\label{ineq:stopadapAIPPM}
\|\hat v_k\|\leq \hat {\rho},
\end{equation}
then return $(\hat z,\hat v)=(\hat z_k,\hat v_k)$;
otherwise, increment $k = k+1$ and go to step 1;

\end{itemize}
\rule[0.5ex]{1\columnwidth}{1pt}

We now give some comments about the above method.
First, it performs two types of iterations, namely, the outer iterations which are indexed by $k$ and the inner ones which are performed by the R-ACG algorithm every time it is called in step~1.
Second, if the call to the R-ACG algorithm in step~1 does not stop with failure then, by Proposition~\ref{prop:nest_complex}(b), the triple $(x,u,\eta)$ output by the R-ACG algorithm together with the stepsize $\lambda$ will satisfy \eqref{eq:ACG_term_descent} where $\widetilde{\phi} = \lam (g + h)$. Hence, by Proposition~\ref{prop:ACG_gd_link}(a), the triple $(\lambda_k, z_k, v_k) := (\lambda, x, u)$ will satisfy \eqref{eq:bd_prox-approx}. If $\lambda$ is also not halved in step~2 then the definition of $\widetilde{M}$ and Proposition~\ref{prop:ACG_gd_link}(b) imply that the triple $(\lambda_k,z_k,v_k)$ also satisfies \eqref{eq:eps_gsm_bd}.
As a consequence, a single iteration of the R-AIPP method implements step 1 of the GD framework.
Third, the termination condition \eqref{ineq:stopadapAIPPM} and Proposition~\ref{prop:approxsol}, with $(\lam, z^-, z, v) = (\lam_k, z_{k-1}, z_k, v_k)$, imply that the required solution, i.e., 
 a pair  $(\hat z, \hat v)$ satisfying \eqref{eq:approx_subgrad}, is obtained when the R-AIPP method terminates. Fourth, since the R-AIPP iterates implement step 1 of GD framework, and the sequence $\{\lam_k\}$ is bounded below (see Lemma~\ref{lem1:adaptiveMet}(b) below), Proposition~\ref{prop:gipp10} implies that the sequence $\{\hat v_k\}$ generated by the R-AIPP method has a subsequence approaching zero, and thus the method must terminate in step 3. Fifth, although the R-AIPP method does not necessarily generate proximal subproblems with convex objective functions, it is shown in  Proposition~\ref{prop:relaxAIPPmethod} below that it has an iteration-complexity similar to that of  the AIPP method of \cite{WJRproxmet1}. Finally, in contrast to the aforementioned AIPP method, the R-AIPP neither requires an upper bound on the quantity $\underline{m}$ in \eqref{eq:m_lower_def} as part of its input nor does it place any restriction on the initial stepsize $\lam_0$.

Each iteration of the R-AIPP method may call  the R-ACG algorithm multiple times (possibly just one time). Invocations of the R-ACG algorithm that stop with success are said to be of type $S$ while the other invocations are said to be of type $O$.
Let $K_S$ (resp., $K_O$) denote the total number of R-ACG calls of type $S$ (resp., type $O$).
The following technical result provides some basic facts about $K_S$, $K_O$ and the sequence of stepsizes $\{\lambda_k\}$.

\begin{lemma}\label{lem1:adaptiveMet} The following statements hold for the R-AIPP method:
\begin{itemize}
\item[(a)] if the stepsize $\lam_{\bar k} \le 1/(2 \underline{m})$ for some $\bar k \geq 1$, then every iteration $k\geq \bar k$ is of type $S$ and, as a consequence, $\lambda_k=\lambda_{\bar k}$ for every $k > \bar k$;
\item[(b)] $K_O$ can be bounded as  $2^{K_O} \leq  \max\{1, 4{\lambda_0} \underline{m} \}$;
\item[(c)] $\{\lambda_k\}$ is non-increasing and satisfies $1 / \lambda_{k} \leq \max\{1/\lam_0, 4\underline{m}\}$ for all $k \geq 1$.
\end{itemize}
\end{lemma}

\begin{proof}
(a) Since $\lambda_{\bar k}\leq 1/(2 \underline{m})$, the definition of $\underline{m}$ in \eqref{eq:m_lower_def} implies that  $\widetilde\phi^{(s)}+\|\cdot-z_{k-1}\|^2/2$ is convex, where  $\widetilde\phi^{(s)}$ is as defined in \eqref{eq:AIPPpsi} with $\lam:=\lam_{\bar k}$. Hence, Proposition~\ref{prop:nest_complex}(c) together with Proposition~\ref{prop:ACG_gd_link}(b) imply that step~1 and step~2 do not halve $\lam$ at the $\bar k^{\rm th}$ iteration, which is to say that this iteration is of type $S$. Since $\{\lambda_k\}$ is clearly nonincreasing, the same conclusion holds true for every iteration $k\geq \bar k$. Moreover, as $\lam$ is not halved for subsequent iterations following $\bar k$, it follows that  $\lambda_k=\lambda_{\bar k}$ for every $k > \bar k$.

(b) Using the fact that immediately before  each iteration of type $O$, the stepsize  $\lambda$ is halved, we see that  the  condition $\lambda_{\bar k}\leq 1/(2m)$ in part (a) would eventually be satisfied for some iteration $\bar k\geq 1$, and hence $K_O$ is finite. Now, note that if $K_O=0$ then the inequality in part (b) follows immediately. Assume then that $K_O\geq 1$. 
It now follows from part (a) and the definition of $K_O$ that $\lam_0 / 2^{K_O - 1} > 1/(2 \underline{m})$, which clearly implies the inequality in part (b).

(c) The first statement follows trivially from the update rule of $\lambda_k$ in the R-AIPP method. Now, note that  the definition of $K_O$ together with the update rule for $\lambda_k$ imply, for every $k\geq 1$, that ${\lambda_0}/{2^{K_O}}\leq \lambda_k.$ 
%If $K_O=0$ then the inequality in part (c) follows from the fact that $\lam_0 \geq 1/(\underline{4m})$
The inequality in part (c) then follows from the inequality in part (b).
\end{proof}

In view of Lemma~\ref{lem1:adaptiveMet}(a), choosing an initial stepsize $\lam_0$ satisfying $\lam_0\leq 1/(2 \underline{m})$ results in an R-AIPP variant with constant stepsize, which resembles the AIPP method described in \cite{WJRproxmet1}. 

% ?????? However, the aforementioned variant does not perform the prescribed constant number of inner iterations (per outer iteration) of the latter one.

The next proposition presents a worst-case iteration complexity bound on the number of inner iterations of the R-AIPP method with respect to the inputs $M, \lam_0,$ and $z_0$, the quantity $\underline{m}$ in \eqref{eq:m_lower_def}, and the  tolerance $\hat\rho$.

\begin{proposition}\label{prop:relaxAIPPmethod}
    Defining $\xi_0 := \max\{1/\lam_0, 4\underline{m}\}$, the R-AIPP method outputs a $\hat\rho$--approximate stationary point $(\hat z,\hat v)$ of \eqref{eq:PenProb2Introa} in at most 
    \begin{equation}\label{eq:r_aipp_compl}
    {\cal O}\left(  \sqrt{M + \xi_0} \left[ \frac{\sqrt{\xi_0}
     \left[\phi(z_{0})-\phi_{*}\right]}{\hat{\rho}^{2}}+\sqrt{\lam_0} \right] \log_1^+\left(\lam_0 \left[M + \xi_0 \right]\right) \right)
    \end{equation}
    inner iterations.
    
\end{proposition}
\begin{proof}
    Let ${\rm TI}_S$ (resp. ${\rm TI}_O$) denote the total number of inner iterations performed during all calls of type $S$
    (resp. type $O$) (see the paragraph preceding Lemma \ref{lem1:adaptiveMet}).
    Clearly, the total number of inner iterations
    is ${\rm TI}:={\rm TI}_S+{\rm TI}_O$.
    We now bound each one of the quantities ${\rm TI}_S$ and ${\rm TI}_O$ separately by using the fact that
    assumption (A2), \eqref{eq:AIPPpsi}, and Proposition~\ref{prop:nest_complex}(a) imply that
    the number of inner iterations performed during each call to the R-ACG algorithm is bounded
    by 
    \begin{equation*}
    \left\lceil\sqrt{2\lam M+1}\log_{1}^{+}\left(C\left[2\lam M+1\right]\right)\right\rceil,
    \end{equation*}
    where $\lam$ is the value of $\lam$ just before the call and
    $C$ is as in \eqref{eq:C_def} with $\widetilde M = \bar \lam M$. 
    
    We first consider ${\rm TI}_O$. Note that Lemma~\ref{lem1:adaptiveMet}(b)   implies that
    $K_O$ is finite. Since ${\rm TI}_O=0$ when $K_O=0$, we may assume without loss of generality that
    $K_O\geq 1$.
    Note that the values of $\lam$ just before the $K_O$ calls of type O are exactly
    $\lam_0,\lam_0/2,\ldots,\lam_0/2^{K_O-1}$.
    Hence, we conclude that
    \begin{align*}
     {\rm{TI}}_{O} & \leq\sum_{i=1}^{K_{O}}\left(\sqrt{\frac{2{\lam_0}M}{2^{i-1}}+1}\right)\log_{1}^{+}\left(C\left[\frac{2{\lam_0}M}{2^{i-1}}+1\right]\right)
     \leq \sum_{i=1}^{K_{O}}\left(\sqrt{\frac{2{\lam_0} \left(M+ \xi_0 \right)}{2^{i-1}}}\right) \log_{1}^{+}\left(2C{\lam_0}\left[M + \xi_0 \right]\right)\\
     & \leq \left(2+\sqrt{2}\right)\sqrt{2{\lam_0}\left(M+\xi_0\right)}\log_{1}^{+}\left(2C{\lam_0}\left[M+\xi_0\right]\right),
    \end{align*}
    where the second inequality is due the fact that Lemma~\ref{lem1:adaptiveMet}(b) implies
    $2^{i-1} \le 2^{K_O-1} \leq 2\lam_0\xi_0$ for every $i \le K_O$.
    %(see the inequality in Lemma~\ref{lem1:adaptiveMet}(a)).
    Thus, we obtain
    \begin{equation}\label{eq:numbinner00}
    {\rm{TI}}_{O}
%    \leq 2\left(2+\sqrt{2}\right)\left(\sqrt{{\lam_0}M}\right)\log_{1}^{+}\left(4C{\lam_0}M\right)
    ={\cal O}\left(\sqrt{{\lam_0}\left(M + \xi_0 \right)}\log_{1}^{+}\left({\lam_0}\left[M+ \xi_0 \right]\right)\right).
    \end{equation}
    
    We now bound ${\rm TI}_S$. Suppose that $K_S > 1$ and observe that  the  termination criterion \eqref{ineq:stopadapAIPPM} is not satisfied in the first $K_S-1$ iterations. Since the R-AIPP method is an instance of the GD framework, it follows from Proposition~\ref{prop:gipp10} that
    \begin{equation} \label{eq:r_aipp_prf_L}
    \hat{\rho}^{2}<\min_{j\leq K_{L}-1}\|\hat{v}_{j}\|^{2}\leq\theta\left(1+2\sqrt{\tau}\right)^{2}\frac{\left[\phi(z_{0})-\phi_{*}\right]}{\sum_{j=1}^{K_{L}-1}\lam_{j}}.
    \end{equation}
    Using the fact that  Lemma~\ref{lem1:adaptiveMet}(c) implies $1/\lam_j \leq \max\{1/\lam_0, 4\bar m\} = \xi_0$ and $\lam_j \leq \lam_0$ for every $j\geq 1$, we obtain 
    \begin{align*}
{\rm TI}_{S} & \leq\sum_{j=1}^{K_{L}}\left(\sqrt{\lam_{j}M+1}\right)\log\left(C\left[\lam_{j}M+1\right]\right) \leq\sum_{j=1}^{K_{L}}\left(\sqrt{\lam_{j}(M+\xi_0)}\right)\log\left(C\lam_{j}[M+\xi_0]\right)\\
& \leq\sqrt{M+\xi_0}\left(\sum_{j=1}^{K_{L}-1}\frac{\lam_{j}}{\sqrt{\lam_{j}}}+\sqrt{\lam_{K_{L}}}\right)\log\left(C{\lam_0}[M+\xi_0]\right)\\
& \leq\sqrt{M+\xi_0}\left[\sqrt{\xi_0}\left(\sum_{j=1}^{K_{L}-1}\lam_{j}\right)+\sqrt{\lam_0}\right]\log\left(C{\lam_0}[M+\xi_0]\right)\\
& \leq\sqrt{M+\xi_0}\left[\theta\left(1+2\sqrt{\tau}\right)^{2}\frac{\sqrt{\xi_0}\left[\phi(z_{0})-\phi_{*}\right]}{\hat{\rho}^{2}}+\sqrt{{\lam_0}}\right] \log\left(C{\lam_0}[M+\xi_0]\right),
    \end{align*}
%    where the last inequality is due to \eqref{eq:r_aipp_prf_L}. Moreover, observe that the assumption that $\underline{m} \leq M$ and the fact that $m_* \leq \underline{m}$ imply $\underline m \leq \max\{m_*, 1/(4\lam_0) \} = {\cal O}(\max\{M, 1/{\lam_0} \})$.
%     and the assumption that $m\leq M$.
%    , and the fact that $1 \leq 2 m \bar \lam \leq 2 M \bar \lam$.
    Hence, we conclude that
    \begin{align}
    {\rm TI}_{S} 
%    & \leq\theta\left(1+2\sqrt{\tau}\right)^{2}\frac{\sqrt{20mM}\log\left(5C\bar{\lam}M\right)\left[\phi(z_{0})-\phi_{*}\right]}{\hat{\rho}^{2}}+\sqrt{5\bar{\lam}M}\log\left(5C\bar{\lam}M\right) \nonumber \\
    & ={\cal O}\left(\sqrt{M + \xi_0}\left[ \frac{\sqrt{\xi_0 }\left[\phi(z_{0})-\phi_{*}\right]}{\hat{\rho}^{2}}+\sqrt{\lam_0} \right]  \log_1^+\left(\lam_0\left[M + \xi_0\right]\right)\right). \label{eq:aux00010}
    \end{align}
    It can be easily seen that the bound in \eqref{eq:aux00010} trivially holds when $K_S\leq1$ in view of the last term in it. Indeed, to prove this, just assume that $\sum_{j=1}^{K_S-1}\lambda_j=0$ in the above argument bounding ${\rm TI}_S$.
    Now,  since ${\rm TI}={\rm TI}_O+{\rm TI}_S$,  the bound in \eqref{eq:r_aipp_compl} follows by adding  \eqref{eq:numbinner00} and  \eqref{eq:aux00010}. 
    
    The last statement of the proposition follows due to Proposition~\ref{prop:approxsol} and the termination condition in step~3 of the R-AIPP method.
\end{proof}

Observe that, unless $\lam_0$ is large or $\underline{m}$ is small, the first term in \eqref{eq:r_aipp_compl} dominates the second one. 
%Second, as $\lam_0$ increases towards $\infty$, the first term decreases towards a constant depending on $\underline{m}$, while the second term in \eqref{eq:r_aipp_compl} increases towards $\infty$. On the other hand, as $\lam_0$ decreases towards $0$, the first term increases towards $\infty$, while the second term in \eqref{eq:r_aipp_compl} decreases towards $0$.

The numerical experiments in Section~\ref{sec:numerical} consider three variants of the R-AIPP method, two of which are R-AIPP instances with different choices of $\lam_0$. More specifically, given an upper bound $m$ on $\underline{m}$, one of the R-AIPP instances chooses $\lam_0 = 0.9/(2m)$ while the other one chooses $\lam_0=1$. For the problem instances considered, the former choice of $\lam_0$ is relatively small, while the latter choice is relatively large.

We now end this section by discussing some possible choices of the initial stepsize $\lam_0$ and how the corresponding R-AIPP instances compare to the AIPP method of \cite{WJRproxmet1}. First,  the AIPP method requires knowledge of an upper bound $m$ on $\underline{m}$ such that $m={\cal O}(M)$, and, as a consequence of a more general iteration complexity bound derived in \cite[Corollary~14]{WJRproxmet1}, its inner iteration complexity can be shown to be
\begin{equation}\label{eq:thcomplexAIPP}
\mathcal{O}\left(\sqrt{M}\left[\frac{\sqrt{m}\left[\phi(z_0) - \phi_*\right]}{{\hat{\rho} }^2}
+\sqrt{\frac{1}{m}}\log^+_1\left(\frac{M}{m}\right)\right]\right).
\end{equation}
Now, if $m$ as above is also known to the R-AIPP and the input $\lam_0$ is set to $1/(4m)$, then its inner iteration complexity \eqref{eq:r_aipp_compl} reduces to 
\begin{equation}\label{eq:r_aipp_spec_compl1}
{\cal O}\left(  \sqrt{M} \left[ \frac{\sqrt{m}
    \left[\phi(z_{0})-\phi_{*}\right]}{\hat{\rho}^{2}} + \sqrt{\frac{1}{m}} \right] \log_1^+\left(\frac{M}{m}\right) \right),
\end{equation}
which is the same as \eqref{eq:thcomplexAIPP} up to a logarithmic factor. On the other hand, if $\lam_0$ is chosen so that $1/\lam_0 = {\cal O}(\underline{m})$ then \eqref{eq:r_aipp_compl} reduces to 
\begin{equation}\label{eq:r_aipp_spec_compl2}
{\cal O}\left(  \sqrt{M} \left[ \frac{\sqrt{\underline{m}}
    \left[\phi(z_{0})-\phi_{*}\right]}{\hat{\rho}^{2}} + \sqrt{\lam_0} \right] \log_1^+\left(\lam_0 M\right) \right),
\end{equation}
whose dominant first term is as good as the dominant first term in \eqref{eq:thcomplexAIPP} whenever $\sqrt{\underline{m}} \log_1^+(\lam_0 M) = {\cal O}(\sqrt{m})$.

\section{A relaxed quadratic penalty AIPP method}\label{sec:penalty}
 
This section presents the R-QP-AIPP method for solving a class of linearly--set--constrained nonconvex composite optimization problems.
Similar to the QP-AIPP method of \cite{WJRproxmet1}, the R-QP-AIPP method is a quadratic penalty--based method that solves a sequence of penalized subproblems, for increasing values of the penalty parameter, using the R-AIPP method of Section~\ref{sec:AIPPmet}.
The section contains two subsections. The first one describes the main problem of interest, its underlying assumptions, and the notion of a corresponding approximate stationary point which R-QP-AIPP method will provably obtain, and briefly outlines a cold--started quadratic penalty--based method for obtaining such a point.
The second one presents a warm--started quadratic penalty--based method, namely, the R-QP-AIPP method, for obtaining the desired stationary point and establishes its ACG iteration complexity.

\subsection{The linearly--set--constrained problem} \label{subsec:lsc_prob}

This subsection describes the main problem of interest in this section, namely, the linearly--set--constrained nonconvex composite optimization problem \eqref{eq:main_lin_prb},  its underlying assumptions, and a notion of an approximate stationary point of it. Moreover, it describes the quadratic penalty subproblem (parameterized a penalty parameter)
associated with \eqref{eq:main_lin_prb}
and discusses the relationship between their corresponding approximate stationary points.
It then outlines a (static and dynamic) cold--started quadratic penalty--based method and its corresponding iteration-complexity bound, which turns out to be larger than that of the QP-AIPP method of \cite{WJRproxmet1}.
%?????? Included in this outline is a discussion of the relationship between an approximate stationary point of the main problem and one of a penalty subproblem.

%This subsection describes the main problem of interest in this section, namely, the linearly-set-constrained problem, outlines
%a (static and dynamic) cold-start quadratic penalty scheme for solving it and discusses its corresponding iteration-complexity bound
%which turn out to be larger than that of the QP-AIPP counterpart.
%In addition, it describes the assumptions made on the problem of interest and a notion of an approximate solution for it whose computation is
% the subject of this section. Moreover, it describes the quadratic penalty subproblems associated with our problem of interest
% and the relationship between their approximate solutions with those of the problem of interest.

%The problem of interest for this section is t
The main problem of interest for this section is the linearly--set--constrained nonconvex composite optimization problem
\begin{equation} \label{eq:main_lin_prb}
\varphi^* := \min \left\{\varphi(z) := f(z) + h(z) : Az \in S, \,  z \in \R^n \right \},
\end{equation}
where closed convex set $S\subseteq \r^p$,  linear operator $A : \rn \mapsto \r^p$,
and functions $f,h:\rn \mapsto (-\infty,\infty]$, satisfy the following assumptions:

%Throughout this section, it is assumed that $\dom h$ is bounded. The  unbounded case is discussed in Section~\ref{sec:concl_remarks}.
%Moreover, 

%We make the following assumptions regarding \eqref{eq:main_lin_prb}:

\begin{itemize}[align=left]
\item [(C1)]  $h\in \bConv$ and its diameter
\begin{equation}
D_h := \sup \{\|z-z'\|:z,z'\in \dom h\} \label{eq:D_h_def}
\end{equation}
is finite; 
\item [(C2)] $A \ne 0$ and
${\cal F} :=\left\{ z \in \dom h : Az\in S \right\}  \ne \emptyset$;
\item [(C3)] $f$ is a nonconvex differentiable function on $\dom h$ and there exist a scalar $L > 0$ such that 
\begin{align}\label{eq:curvature_penf}
\|\nabla f(z)-\nabla f(u)\| \leq L\|z-u\| \quad \forall u,z\in \dom h;
%\quad f(u) \geq \ell_f (u;z) - \frac{m}{2} \|u-z\|^2;
\end{align}
\item [(C4)] ${\varphi_0^*} := \inf\{\varphi(z):z\in\rn\} > -\infty$.
% and there exists $\hat c\geq 0$ such that  $\hat \varphi_{\hat c} > -\infty$, where $\hat \varphi_c$ is as in \eqref{eq:pen_sub}.

\end{itemize}

We make two remarks about the above assumptions.
First, Lemma~\ref{lem:lsc} in Appendix~\ref{app:penalty} shows that (C1), (C3), and the additional assumption that $f$ be lower semicontinuous on $\cl(\dom h)$ imply (C4). Second, denoting $\underline{m}$ as the quantity \eqref{eq:m_lower_def} with $g=f$, assumption (C3) implies that $\underline{m} \in (0,L]$. Moreover, it is shown in Theorem~\ref{thm:rqp_aipp_compl} below that the smaller $\underline{m}$ is, the better the iteration complexity of the R-QP-AIPP method becomes.

We now discuss a notion of approximate stationary point for \eqref{eq:main_lin_prb}.
 Clearly, \eqref{eq:main_lin_prb} is equivalent to the problem 
\begin{equation} \label{eq:block_penalty}
\min \left\{ f(z) + h(z) : Az - s = 0, \, s\in S, \,  z \in \R^n \right\}.
\end{equation}
%--------
%
%Assume first that
%\[
%\bar z  \notin \cl(\dom h).
%\]
%Then, there exists $\epsilon>0$ such that $\bar B(\bar z; \epsilon) \cap \cl(\dom h) = \emptyset$,
%and
%hence $(f+h)(z) = + \infty$ for all $z \in \bar B(\bar z; \epsilon)$. So,
%\[
%\liminf_{z \to \bar z} (f+h)(z) = \infty = (f+h)(\bar z)
%\]
%Now assume that 
%\[
%\bar z  \in \cl(\dom h).
%\]
%%If $\bar z \notin \dom h$ then
%%\[
%%\liminf_{z \to \bar z} h(z) = h(\bar z) = \infty.
%%\]
%Hence,
%\[
%\liminf_{z \to \bar z} (f+h)(z) \ge \liminf_{z \to \bar z} f(z)  + \liminf_{z \to \bar z} h(z) = (f+h)(\bar z)
%\]
%If $\bar z \in \dom h$ then
%
%
%---------
Moreover, a necessary condition for a point $(\hat z, \hat s)\in\dom h\times S$ to be a local minimum to the above problem is that  there exists a multiplier $\hat q\in \r^p$ such that
\begin{equation}\label{eq:alt_optimality}
0 \in \nabla f(\hat z) + \partial h(\hat z) + A^* \hat{q}, \quad A\hat z - \hat s = 0, \quad \hat{q} \in N_S(\hat s), \quad \hat s\in S.
\end{equation}
Given a tolerance pair $(\hat \rho,\hat \eta)\in \Re^2_{++}$, a triple 
$([\hat z, \hat s], \hat{q}, \hat v) \in [\dom h \times S] \times \R ^p \times \R^n$ is said to be a $(\hat \rho,\hat \eta)$--approximate stationary point of  \eqref{eq:probintroa} if it satisfies
\begin{equation}\label{eq:approx_cond_lin_constr}
\hat v \in \nabla f(\hat z) + \partial h(\hat z) + A^* \hat{q},  \quad \|\hat v\| \le \hat \rho, \quad \|A\hat z - \hat s\| \le \hat \eta, \quad \hat{q}\in N_S(\hat s), \quad \hat s\in S.
\end{equation}
Clearly, a $(\hat \rho,\hat \eta)$--approximate stationary point $([\hat z, \hat s], \hat{q}, \hat v)$ of \eqref{eq:main_lin_prb} when $(\hat \rho,\hat \eta) = (0,0)$ means
that the pair $(\hat z, \hat s)$ and the multiplier $\hat q$ satisfy \eqref{eq:alt_optimality}.

%?????Our goal for the remainder of this section is to describe the R-QP-AIPP method for finding a $(\hat \rho,\hat \eta)$--approximate stationary point of \eqref{eq:main_lin_prb}. 
%??? Our goal for the remainder of this subsection is to describe a quadratic penalty--based method for finding a $(\hat \rho,\hat \eta)$--approximate stationary point of \eqref{eq:main_lin_prb}.
%the relationship between approximate stationary points of \eqref{eq:main_lin_prb} and those of a related sequence of quadratic penalty subproblems. 
%Before proceeding, we first describe the penalty subproblems, 

We now describe the quadratic penalty subproblem (parameterized by a penalty parameter) with respect to
\eqref{eq:main_lin_prb}. Defining the quadratic penalty function $p_S:\r^p\mapsto\r_{+}$ as
\begin{equation} \label{eq:p_S_def}
p_S(x) := \frac{1}{2} \|x - \Pi_S(x)\|^2
\end{equation} 
where 
\begin{equation} \label{eq:Pi_S}
\Pi_S(x) := \argmin\{\|u - x\| : u \in S\}
\end{equation}
for every $x \in \r^p$,
the quadratic penalty subproblem parameterized by a  penalty parameter $c>0$ with respect to \eqref{eq:main_lin_prb}  is
\begin{equation} \label{eq:pen_sub}
{\varphi}_{c}^*  := \min \left \{ \varphi_c(z) := \varphi(z) + c \cdot p_S(Az) : z \in \Re^n \right \}.
\end{equation}
%makes use of 
%Second, in contrast to \eqref{eq:pen_sub_intro}, this method will solve a sequence of the penalty subproblems of the form

We now make four remarks regarding \eqref{eq:pen_sub}. First, \eqref{eq:pen_sub_intro} is an instance of \eqref{eq:pen_sub} in which $S=\{b\}$. Second, when $c=0$, the optimal value of \eqref{eq:pen_sub} coincides with
$\varphi^*_0$ in (C4), and hence there is no abuse of notation made here.
Third, it is easily seen that
\begin{equation} \label{eq:pen_topo}
\varphi^* \geq \varphi_{\bar c}^* \geq \varphi_{c}^* \quad \forall \bar c > c \ge 0,
\end{equation}
where $\varphi^*$ is as in \eqref{eq:main_lin_prb}. Finally, \eqref{eq:pen_sub} is a penalty subproblem
 involving only the original variable $z$ of formulation \eqref{eq:main_lin_prb}
rather than the one associated with \eqref{eq:block_penalty}
(constructed as in Section~\ref{sec:intro} with $Az=b$ replaced by $A z - s =0$),
which  involves the pair of variables $(z, s)$.
%of
%the equivalent one \eqref{eq:block_penalty}, namely, the one as in \eqref{eq:pen_sub_intro}
%with $Az=b$ replaced by $A\hat z - \hat s =0$.
%Since \eqref{eq:pen_sub} has less variables, we believe that it might be more efficient from a computational
%point of view.

The following result shows how a $\hat\rho$--approximate stationary point of \eqref{eq:pen_sub} yields a $(\hat \rho, \hat \eta)$--approximate stationary point of \eqref{eq:main_lin_prb} when the penalty parameter $c$ is sufficiently large. 

\begin{proposition} \label{prop:pen_termination}
    Let $(\hat{\rho}, \hat \eta)\in \R_{++}^2$ and $c\geq 0$ be given and suppose that $(\hat{z},\hat{v})$
    is a $\hat\rho$--approximate stationary point of \eqref{eq:pen_sub} as in \eqref{eq:approx_subgrad} with $g=f + c \cdot (p_S \circ A)$. Moreover, let $\underline m$ be as in \eqref{eq:m_lower_def} with $g=f$ and define
    \begin{gather} 
    g_c := f + c \cdot (p_S \circ A), \quad
    M_c := L + c\|A\|^2, \label{eq:pen_aipp_input}
    \\
    \hat s := \Pi_S(A\hat z), \quad
    \hat q := c (A\hat z - \hat s), \quad
    T := 2\left({\varphi}^*-\varphi^*_0+\hat{\rho}D_{h}\right) + \underline{m} D_{h}^{2}, \label{eq:pen_c_def}
    \end{gather}
    where ${\varphi}^*$ and ${\varphi}^*_0$ are as in \eqref{eq:main_lin_prb} and (C4), respectively. Then,
    the following statements hold:
    \begin{itemize}
        \item[(a)] for every $u,z\in\dom h$, the pair $(g,M)=(g_c, M_c)$ satisfies \eqref{eq:curvature};
        \item[(b)] the triple $([\hat z, \hat s], \hat q, \hat v)$ satisfies the inclusions and the first inequality of \eqref{eq:approx_cond_lin_constr} and 
        \begin{equation}
        \frac{c}{2}\|A\hat{z} - \hat s\|^{2}\leq{\varphi}^*-\varphi(\hat{z})+\hat{\rho}D_{h}+\frac{1}{2} \left(\underline{m} D_{h}^{2}\right) ;\label{eq:pen_feas_bd}
        \end{equation}
        \item[(c)] if, in addition, the penalty parameter $c$ satisfies
        \begin{equation}\label{eq:c_bd}
        c\geq \frac{T}{\hat{\eta}^{2}},
        \end{equation} then $\|A\hat{z} - \hat s \|\leq\hat{\eta}$,
        and hence $([\hat z, \hat s], \hat q, \hat v)$ is a $(\hat \rho,\hat \eta)$--approximate stationary point of \eqref{eq:main_lin_prb}.
    \end{itemize}
\end{proposition}

\begin{proof}
    Throughout this proof, we will make use of the well known fact (see, for example, \cite[Theorems 6.39 \& 6.60]{beck2017first}) that $p_S$ is convex, differentiable, its gradient is $1$--Lipschitz, and, for every $x\in\R^p$, 
    \begin{equation} \label{eq:p_S_props}
    \nabla p_S(x)=x-\Pi_S(x)\in N_{S}\left(\Pi_{S}(x)\right).
    \end{equation}
    
    (a) This follows immediately from the definition of $g_c$ in \eqref{eq:pen_aipp_input}, assumption (C3), and the fact that $\nabla p_S$ is 1--Lipschitz continuous.    

    (b) Using the definitions of $\hat q$ and $\hat s$ given in  \eqref{eq:pen_c_def}, and the fact that \eqref{eq:p_S_props} at $x=A\hat z$ implies $ c \nabla p_S(A\hat z) = \hat q$, observe that: (i) $c \nabla (p_S \circ A) \hat z = c A^* \nabla p_S(A\hat z) = A^* \hat q$; 
    and (ii) $\hat q \in N_S(\hat s)$. It now follows from the definition of a $\hat\rho$--approximate stationary point of \eqref{eq:pen_sub} with $g=f + c \cdot (p_S \circ A)$ and the previous observations that 
    \begin{align} \label{eq:grad_p_S_incl}
    \hat{v}  \in \nabla f(\hat{z})+\partial h(\hat{z})+ c \nabla (p_S \circ A) \hat z
    &  = \nabla f(\hat{z})+\partial h(\hat{z})+ A^* \hat q \nonumber \\ 
    & \subseteq \nabla f(\hat{z})+\partial h(\hat{z})+A^{*}N_{S}(\hat{s}).
    \end{align}
    Hence, with the additional fact  that $\|\hat{v}\|\leq\hat{\rho}$ from \eqref{eq:approx_subgrad}, it follows that the inclusions and first inequality of \eqref{eq:approx_cond_lin_constr}
    hold. Next, observe that the convexity of $p_{S}$ and the first inclusion in \eqref{eq:grad_p_S_incl} imply that
    $ \hat{v}\in\nabla f(\hat{z})+\partial\left[h + c \cdot (p_S \circ A)\right](\hat{z})$
    or equivalently, 
    \begin{equation}
    h(u)+c \cdot p_{S}(Au)\geq h(\hat{z})+c \cdot p_{S}(A\hat{z})+\left\langle \hat{v}-\nabla f(\hat{z}),u-\hat{z}\right\rangle \quad\forall u\in\dom h.\label{eq:pen_incl}
    \end{equation}
    Considering \eqref{eq:pen_incl} at any $u \in {\cal F}$ and using the fact that $p_S(Au)=0$ for any $u\in {\cal F}$, the definition of $\underline{m}$ in \eqref{eq:m_lower_def}, and the definitions
    of $p_{S}$ and $\hat s$, we conclude that
    \begin{align*}
    \frac{c}{2}\|A\hat{z} - \hat s \|^{2} & \leq h(u)-h(\hat{z})+\left\langle \nabla f(\hat{z}),u-\hat{z}\right\rangle -\left\langle \hat{v},u-\hat{z}\right\rangle \\
    & \leq(f+h)(u)-(f+h)(\hat{z})+\|\hat{v}\|\|u-\hat{z}\|+\frac{1}{2}\left( \underline{m} \|u-\hat{z}\|^{2} \right)\\
    & \leq{\varphi}(u)-\varphi(\hat{z})+\hat{\rho}D_{h}+\frac{1}{2} \left(\underline{m} D_{h}^{2}\right).
    \end{align*}
    Taking the infimum over $u\in {\cal F}$ immediately implies \eqref{eq:pen_feas_bd}.
    
    (c) Using \eqref{eq:c_bd}, the fact that $\varphi(\hat{z})\geq{\varphi}^*_0$
    , and the definition of $T$, it follows from part (b) that
    \[
    \|A\hat{z} - \hat s\|^{2} \leq \frac{1}{c} \left[ 2\left({\varphi}^*-\varphi^*_0 +\hat{\rho}D_{h}\right)+ \underline{m} D_{h}^{2} \right]=\frac{T}{c}\leq\hat{\eta}^{2}.
    \]
\end{proof}

In view of the above proposition, we now outline a static penalty method for obtaining a $(\hat \rho, \hat \eta)$--approximate
stationary point of \eqref{eq:main_lin_prb}. First, let $z_0 \in \dom h$ be given and  select a penalty parameter $c = {\cal O}(\hat \eta ^{-2})$ satisfying \eqref{eq:c_bd}. Second, obtain a $\hat \rho$--approximate stationary point $(\hat z, \hat v)$ of \eqref{eq:pen_sub} using the R-AIPP method of Section~\ref{sec:AIPPmet} with starting point $z_0$ and inputs $M=M_c$ and $(g,h)=(g_c,h)$, which satisfy assumptions (A1)--(A3) in view of Proposition~\ref{prop:pen_termination}(a) and assumptions $(C1)$ and $(C3)$. Finally, compute the pair $(\hat s, \hat q)$ according to \eqref{eq:pen_c_def} and output the triple $([\hat z, \hat s], \hat q, \hat v)$, which is a $(\hat \rho,\hat \eta)$--approximate stationary point of \eqref{eq:main_lin_prb} in view of Proposition~\ref{prop:pen_termination}(c).
Using \eqref{eq:pen_aipp_input} with $(c, \bar c) = (0,c)$, the definitions in \eqref{eq:pen_aipp_input}, the fact that $c = {\cal O}(\hat\eta ^{-2})$,  and the complexity bound for the
R-AIPP method described in Proposition~\ref{prop:relaxAIPPmethod} with $M=M_c$, it is easy to see that the ACG iteration complexity of the outlined method is
\begin{equation} \label{eq:naive_qp_compl}
{\cal O}\left(\sqrt{M_c + \xi_0}\left[\frac{\sqrt{\xi_0}\left[\varphi_{c}(z_{0})-\varphi^*_{0}\right]}{\hat{\rho}^{2}}+\sqrt{\lam_0}\right]\log_{1}^{+}\left(\lam_0 \left[ M_{c}  + \xi_0 \right]\right)\right)={\cal O}\left(\hat{\rho}^{-2}\hat{\eta}^{-3}\log_{1}^{+}\hat{\eta}^{-1}\right),
\end{equation}
where $\xi_0 := \max\{1/\lam_0, 4\underline{m}\}$ and the last quantity ignores any constants aside from the tolerances.
%(Note that the above bound is actually
%${\cal O}\left(\hat{\rho}^{-2}\hat{\eta}^{-1}\log_{1}^{+}\hat{\eta}^{-1}\right)$ when $z_0 \in {\cal F}$
%but our interest lies in the case where $z_0 \notin {\cal F}$ since usually a point $z_0 \in {\cal F}$ is not known.)
A drawback of this static penalty method is that it requires in its first step the selection of a single parameter $c$, which is generally difficult to obtain. This issue can be circumvented by 
considering a dynamic cold--started penalty method in which the static  penalty method is repeated for a sequence of increasing values of $c$ and common starting point $z_0$. It can be shown that the resulting cold--started dynamic penalty method has an ACG iteration complexity that is still on the same order as \eqref{eq:naive_qp_compl}.
Note that the bound \eqref{eq:naive_qp_compl} is actually
${\cal O}(\hat{\rho}^{-2}\hat{\eta}^{-1}\log_{1}^{+}\hat{\eta}^{-1})$ when $z_0 \in {\cal F}$
(see (C2))
but our interest lies in the case where $z_0 \notin {\cal F}$ since an initial point $z_0 \in {\cal F}$ is generally not known.

The QP-AIPP method of \cite{WJRproxmet1} is a modified cold--started dynamic penalty method like the one just outlined, but which replaces the R-AIPP method called in step~2 of the static penalty method with the AIPP method of \cite{WJRproxmet1}. It has been shown in \cite[Theorem 18]{WJRproxmet1} that its ACG iteration complexity bound for finding a $(\hat\rho, \hat\eta)$--approximate stationary point of \eqref{eq:probintroa} is ${\cal O}(\hat \rho ^{-2} \hat \eta^{-1})$. This bound is established without assuming that $\dom h$ is bounded and is clearly better than the one in \eqref{eq:naive_qp_compl}.

The next subsection considers a warm--started dynamic penalty method, similar to the one described immediately after Proposition~\ref{prop:pen_termination}, in which the input $z_0$ to the R-AIPP call for solving the next penalty subproblem is chosen to be the output $\hat z$ from the R-AIPP call for solving the current one. It is shown in Theorem~\ref{thm:rqp_aipp_compl} of Subsection~\ref{subsec:r_qp_aipp} that its ACG iteration complexity is ${\cal O}(\hat \rho ^{-2} \hat \eta^{-1} \log_1^+ \hat \eta ^{-1})$,  which is the same as the one for the QP-AIPP method up to a logarithmic factor. As a side remark, we note that although a warm--started version of the QP-AIPP method in \cite{WJRproxmet1}
can be also considered, the aforementioned ${\cal O}(\hat \rho ^{-2} \hat \eta^{-1})$ ACG iteration complexity bound was derived
for its cold--started version.

\subsection{The R-QP-AIPP method}\label{subsec:r_qp_aipp}

The goal of this subsection is to describe the R-QP-AIPP method, i.e.,  the warm--started dynamic penalty method mentioned at the end of Subsection~\ref{subsec:lsc_prob}, and establish its corresponding ACG iteration complexity.

We start by describing the R-QP-AIPP method.
% the dynamic penalty scheme mentioned in the previous paragraph.

\noindent\begin{minipage}[t]{1\columnwidth}
\noindent\rule[0.5ex]{1\columnwidth}{1pt}

\noindent \textbf{R-QP-AIPP method.}

\noindent\rule[0.5ex]{1\columnwidth}{1pt}
\end{minipage}

\vspace*{0.5em}

\noindent {\bf Input:} a problem instance of the form in \eqref{eq:main_lin_prb}, a scalar $L>0$, a tolerance pair $(\hat\rho, \hat\eta) \in \r_{++}^2$, an initial point $\hat{z}_0 \in \dom h$, a scalar $\lam_0 > 0$, and a pair of parameters  $(\theta,\tau) \in   (2,\infty)\times (0, \infty)$;

\vspace*{1em}

\noindent {\bf Output:} a triple $([\hat z, \hat s], \hat q, \hat v) \in [\dom h \times S] \times \r^p \times \rn $ satisfying \eqref{eq:approx_cond_lin_constr};

\begin{itemize}
\item [(0)] set $c_0:=L/\|A\|^2$ and $l = 1$;
\item [(1)] set $(c, z_0) := (c_{l-1}, \hat{z}_{l-1})$ and 
\begin{equation*}
M_{c} := L + c\|A\|^2, \quad g_{c} := f + c \cdot (p_S \circ A);
\end{equation*}
call the R-AIPP method on \eqref{eq:PenProb2Introa} with inputs $\hat \rho$, $M_{c}$, $(g_{c},h)$, $z_0$, $\lam_0$, and $(\theta, \tau)$, 
to obtain a $\hat \rho$-approximate stationary point  $(\hat z,\hat v)$ of \eqref{eq:PenProb2Introa}, and set 
\[
(\hat z_l, \hat v_l) = (\hat z, \hat v),\quad \hat s_l = \Pi_S (A \hat z_l), \quad \hat q_l = c (A\hat z_l - \hat s_l);
\]
\item [(2)]
if the residual
\[
\|A\hat z_l - \hat s_l \| \leq \hat\eta,
\] then return $([\hat z, \hat s], \hat q, \hat v) = ([\hat z_l, \hat s_l], \hat q_l, \hat v_l)$; otherwise, set  $c_l = 2 c_{l-1}$, increment $l = l +1$, and go to step 1.
\end{itemize} 
\noindent\rule[0.5ex]{1\columnwidth}{1pt}

\gap 

Before giving some remarks about the above method, we discuss its general structure. Every loop of the  R-QP-AIPP method  invokes in its step 1 the R-AIPP  method of Section~\ref{sec:AIPPmet} to compute a
$\hat \rho$-approximate stationary point
of the current penalty subproblem \eqref{eq:pen_sub}. The latter method in turn uses the R-ACG algorithm of Section~\ref{sec:Nesterov's-Method} as a subroutine in its implementation
(see step~1 of the R-AIPP method). Moreover, step 1 of the R-QP-AIPP implements a warm--start strategy, namely, the input point $z_0$ of the current  R-AIPP call is set to be the output point $\hat z_{l-1}$ of the previous R-AIPP call.

We now make three remarks about the R-QP-AIPP method. First, it follows from Proposition~\ref{prop:pen_termination}(b) that, for every $l\geq 1$, the triple $([\hat z, \hat s], \hat q, \hat v) = ([\hat z_l, \hat s_l], \hat q_l, \hat v_l)$ satisfies the inclusions and the first inequality in \eqref{eq:approx_cond_lin_constr}.
Second, since every loop of the R-QP-AIPP method doubles $c$, the condition \eqref{eq:c_bd} will be eventually satisfied. Hence, in view of Proposition~\ref{prop:pen_termination}(c), the pair $(\hat z, \hat s)$ corresponding to this $c$ will satisfy the condition $\|A\hat z - \hat s\| \le \hat \eta$ and the
R-QP-AIPP method will stop in view of its stopping criterion in step 2.
Finally, in view of the first and second remarks, we conclude that the R-QP-AIPP method outputs a triple $([\hat z, \hat s], \hat q, \hat v)$ satisfying \eqref{eq:approx_cond_lin_constr}.

Before deriving the ACG iteration complexity of the R-QP-AIPP method, we note that the number of ACG iterations needed in the $(l+1)^{\rm th}$ execution of its step 1 depends on the quantity $\varphi_{c_{l}}(\hat{z}_{l})-{\varphi}^*_{c_{l}}$ (see the left--hand--side of \eqref{eq:naive_qp_compl} with $(c,z_0)=(c_{l}, \hat{z}_{l})$). The result below shows that the warm--start strategy in step 1 of the method together with the boundedness of $\dom h$ imply that the aforementioned quantity has an upper bound that is independent of the size of the parameter $c_{l}$.

\begin{lemma} \label{lem:poten_bd}
    Let $c_0$ and $\hat{z}_0$ be as in step 0 and the input of the R-QP-AIPP method, respectively, and define 
    \begin{equation}
    S_0:=\varphi_{c_{0}}(\hat{z}_{0})-\varphi^*_{c_{0}},\quad Q_0 := T +  S_0,\label{eq:SQ_def}
    \end{equation}
    where  ${\varphi}^*_{c}$ and $T$ are as in \eqref{eq:pen_sub} and \eqref{eq:pen_c_def}, respectively. Then, for every $l \geq 0$, we have  
    \begin{align} \label{eq:poten_bd}
    \varphi_{c_{l}}(\hat{z}_{l})-{\varphi}^*_{c_{l}}\leq Q_0.
    \end{align}
\end{lemma}

\begin{proof}
The case in which $l=0$ follows trivially from the definition of $S_0$ in (\ref{eq:SQ_def}). Consider now the case in which $l\geq1$. Remark that $c_l = 2c_{l-1}$ due to step 2 of R-QP-AIPP and \eqref{eq:pen_sub} and that $(\hat{z}_{l}, \hat{v}_{l})$ is a $\hat\rho$--approximate stationary point of \eqref{eq:pen_sub} with $c=c_{l-1}$ due to the warm--start strategy in  step 1 of the R-QP-AIPP method. It now follows from the aforementioned remarks, the last inequality in \eqref{eq:pen_topo} with $c = c_l$, and Proposition~\ref{prop:pen_termination}(b) with
$(\hat{z},c)=(\hat{z}_{l},c_{l-1})$, that
\begin{align}
\varphi_{c_{l}}(\hat z_{l})-{\varphi}^*_{c_{l}} & \leq\varphi_{c_{l}}(\hat z_{l})-{\varphi}^*_0 = \varphi(\hat z_{l})+
2\left[\frac{c_{l-1}}{2}\|A\hat{z}_{l} - \hat s_{l}\|^2\right] - {\varphi}^*_0 \nonumber \\
&\leq\varphi(\hat z_{l}) + 2\left[{\varphi}^*-\varphi(\hat{z}_{l})+\hat{\rho}D_{h}+\frac{1}{2} \left(\underline{m} D_{h}^{2}\right)\right] -{\varphi}^*_0. \label{eq:poten_bd1}
\end{align}
Grouping terms in the last expression together, using the definition of $Q_0$ given in \eqref{eq:SQ_def}, and the fact that $\varphi(\hat z_l) \geq \varphi_0^*$, we conclude that
\begin{align}
\varphi(\hat z_{l}) + 2\left[{\varphi}^*-\varphi(\hat{z}_{l})+\hat{\rho}D_{h}+\frac{1}{2} \left(\underline{m} D_{h}^{2}\right) \right] -{\varphi}^*_0 \leq 2\left({\varphi}^* - {\varphi}^*_0 +  \hat{\rho}D_{h}\right)+{\underline{m} D_{h}^{2}} = T \leq Q_0. \label{eq:poten_bd2}
\end{align}    
Combining \eqref{eq:poten_bd1} and \eqref{eq:poten_bd2} yields \eqref{eq:poten_bd}.
\end{proof}

The following result establishes the iteration complexity of the R-QP-AIPP method with respect to the inputs $L, \lam_0,$ and $z_0$, the quantity $\underline{m}$ in \eqref{eq:m_lower_def} with $g=f$, and the tolerance pair $(\hat\rho, \hat\eta)$.

\begin{theorem} \label{thm:rqp_aipp_compl}
    Given a tolerance pair $(\hat{\rho},\hat{\eta})\in\r_{+}^{2}$,  define
    \begin{equation}
    \Xi_{\hat \eta} := L + \frac{T\|A\|^{2}}{\hat{\eta}^{2}} , 
    \label{eq:pen_compl_const}
    \end{equation}
    where $T$ is given in \eqref{eq:pen_c_def}. Then, defining $\xi_0 := \max\{1/\lam_0, 4\underline{m}\}$, the R-QP-AIPP method outputs a $(\hat\rho, \hat\eta)$--approximate stationary point $([\hat z, \hat s], \hat q, \hat v)$ of \eqref{eq:main_lin_prb} in at most
    \begin{equation}
    {\cal O}\left(\sqrt{\Xi_{\hat \eta} + \xi_0}\left[\frac{\sqrt{\xi_0} Q_0}{\hat{\rho}^{2}}+\sqrt{\lam_0}\right] \log_1^+\left({\lam_0} \left[\Xi_{\hat \eta} + \xi_0\right]\right)\right),
    \label{eq:r_qp_aipp_compl}
    \end{equation}
    ACG iterations, where  $Q_0$ is as in \eqref{eq:SQ_def}.
\end{theorem}

\begin{proof}
Define $T_{\hat \eta} := T / \hat \eta ^ 2$ and let $\bar{l}$ be the smallest index such that
$c_{\bar{l}-1}\geq T_{\hat{\eta}}$. Since the R-QP-AIPP invokes the R-AIPP method with $(M, g ) = (M_{c_{l-1}}, g_{c_{l-1}})$, it follows from Lemma~\ref{lem:poten_bd} and Proposition~\ref{prop:relaxAIPPmethod}, with $M=M_{c_{l-1}}$,
that the total number of ACG iterations at the $l^{\rm th}$ iteration of
the R-QP-AIPP method is on the order of
\begin{equation}\label{eq:qp_compl_stage}
{\cal O}\left(\sqrt{M_{c_{l-1}} + \xi_0}\left[\frac{ \sqrt{\xi_0} Q_0}{\hat{\rho}^{2}}+\sqrt{\lam_0}\right]\log_1^+\left({\lam_0} \left[M_{c_{l-1}} + \xi_0\right]\right)\right).
\end{equation}
Hence, the R-QP-AIPP method stops in a total number of ACG iterations bounded above by the sum of the quantity in \eqref{eq:qp_compl_stage} over $l=1,\ldots,\bar{l}$.

We now focus on simplifying some of the quantities in the aforementioned sum. Using the fact that $L=c_{0}\|A\|^{2}$, we obtain the bound
\begin{equation} \label{eq:pen_curv_bd}
M_{c_{l-1}}=L+c_{l-1}\|A\|^{2}=L+2^{l-1}c_{0}\|A\|^{2}\leq2^{l-1}\left(L+c_{0}\|A\|^{2}\right)= 2^{l}c_{0}\|A\|^{2}.
\end{equation}
Now, if $\bar l = 1$, then the above inequality implies that $M_{c_{\bar l-1}} \leq 2 c_0 \|A\|^2 = 2 L = {\cal O}\left(\Xi_{\hat \eta}\right)$.
Assume then that $\bar l\geq 2$. Observe that the definition of $\bar{l}$ implies that $2^{\bar{l}-1}c_{0}\leq2T_{\hat{\eta}}$ or, equivalently, $\sqrt{c_{0}}\sqrt{2}^{\bar{l}}\leq2\sqrt{T_{\hat{\eta}}}$.
Combining the previous inequality with \eqref{eq:pen_curv_bd}, we conclude that
\begin{align}
& \sum_{k=1}^{\bar{l}} \sqrt{M_{c_{l-1}} + \xi_0 } 
\leq \sum_{k=1}^{\bar{l}}\sqrt{2^{l}c_{0}\|A\|^{2} + \xi_0} \leq\sqrt{2}^{\bar{l}}\left(1+\sqrt{2}\right)\sqrt{2c_{0}\|A\|^2 + \xi_0} \nonumber \\
& \le 8\sqrt{\|A\|^{2}T_{\hat{\eta}} + \xi_0}={\cal O}\left(\sqrt{\Xi_{\hat \eta} + \xi_0}\right)\label{eq:lam_Mk_bd1}
\end{align}
and also 
\begin{align}
&\log_{1}^{+}\left(M_{c_{l-1}} + \xi_0 \right) 
\leq\log_{1}^{+}\left(2^{\bar{l}}c_{0}\|A\|^{2} + \xi_0 \right) 
\leq \log_{1}^{+}\left(4T_{\hat{\eta}}\|A\|^{2}+ \xi_0 \right)
= {\cal O}\left(\log_{1}^{+}\left[\Xi_{\hat \eta} + \xi_0 \right]\right).\label{eq:lam_Mk_bd2}
\end{align}
It now follows from \eqref{eq:qp_compl_stage}, \eqref{eq:lam_Mk_bd1}, and \eqref{eq:lam_Mk_bd2} that the R-QP-AIPP method stops in a total number of
ACG iterations bounded by the quantity in \eqref{eq:r_qp_aipp_compl}.
    
The statement that $([\hat z, \hat s], \hat q, \hat v)$ is a $(\hat \rho, \hat \eta)$--approximate stationary point follows from Proposition~\ref{prop:pen_termination}(b) and the termination condition in step~2 of the R-QP-AIPP method.
\end{proof}

We now make three remarks about the complexity bound in \eqref{eq:r_qp_aipp_compl}.
First, in terms of the tolerance pair $(\hat \rho, \hat \eta)$, it is ${\cal O}(\hat \rho ^{-2} \hat \eta^{-1} \log_1^+ \hat \eta ^{-1})$, which improves upon the complexity in \eqref{eq:naive_qp_compl} by a $\Theta (\hat \eta ^{-2})$ factor. 
Second, unless $\lam_0$ is large or $\underline{m}$ is small, the first term in \eqref{eq:r_qp_aipp_compl} dominates the second one. 
%Third, as $\lam_0$ increases towards $\infty$, the first term decreases towards a constant depending on $\underline{m}$, while the second term in \eqref{eq:r_qp_aipp_compl} increases towards $\infty$. On the other hand, as $\lam_0$ decreases towards $0$, the first term increases towards $\infty$, while the second term in \eqref{eq:r_qp_aipp_compl} decreases towards $0$.}    
    
We now end this section by discussing some possible choices of the initial stepsize $\lam_0$ and how the corresponding R-QP-AIPP instances compare to the QP-AIPP method of \cite{WJRproxmet1}.
First, recall that the QP-AIPP method requires the knowledge of an upper bound $m$ on $\underline{m}$ such that $m={\cal O}(L)$, and remark that, under the same assumptions of this paper, it can be shown using \cite[Theorem~18]{WJRproxmet1} that its ACG iteration complexity is
\begin{align}\label{eq:qp_aipp_compl}
{\cal O}\left(\sqrt{\Xi_{\hat\eta}}\left[\frac{ \sqrt{m} {Q}_0} {\hat{\rho}^2}+\sqrt{\frac{1}{m}}\log_{1}^{+}\left(\frac{\Xi_{\hat\eta}}{m}\right)\right]\right).
\end{align}
Now, if $m$ as above is also known to the R-AIPP and the input $\lam_0$ is set to $1/(4m)$, then the ACG iteration complexity \eqref{eq:r_qp_aipp_compl} reduces to 
\begin{equation}\label{eq:r_qp_aipp_spec_compl1}
{\cal O}\left(  \sqrt{\Xi_{\hat\eta}} \left[ \frac{\sqrt{m} Q_0} {\hat{\rho}^{2}} + \sqrt{\frac{1}{m}} \right] \log_1^+\left(\frac{\Xi_{\hat\eta}}{m}\right) \right),
\end{equation}
which is the same as \eqref{eq:r_qp_aipp_compl} up to a logarithmic factor. On the other hand, if $\lam_0$ is chosen so that $1/\lam_0 = {\cal O}(\underline{m})$ then \eqref{eq:r_qp_aipp_compl} reduces to 
\begin{equation}\label{eq:r_qpaipp_spec_compl2}
{\cal O}\left(  \sqrt{\Xi_{\hat\eta}} \left[ \frac{\sqrt{\underline{m}} Q_0} {\hat{\rho}^{2}} + \sqrt{\lam_0} \right] \log_1^+\left(\lam_0 \Xi_{\hat\eta}\right) \right),
\end{equation}
whose dominant first term is as good as the dominant first term in \eqref{eq:qp_aipp_compl} when $\sqrt{\underline{m}} \log_1^+(\lam_0 \Xi_{\hat\eta}) = {\cal O}(\sqrt{m})$.

\section{Numerical experiments\label{sec:numerical}}

This section presents computational results that highlight the performance of the R-AIPP and R-QP-AIPP methods. It contains three subsections. The first subsection compares three variants of the R-AIPP method against three state-of-the-art nonconvex composite optimization algorithms. The second subsection uses the six algorithms in the first subsection as subroutines in a quadratic penalty method similar to the one in Section~\ref{sec:penalty}. More specifically, given an algorithm $A$ out of the six algorithms in the first subsection, a corresponding quadratic penalty method is considered in which steps~0 to 2 of the R-QP-AIPP method in Section~\ref{sec:penalty} are executed with algorithm $A$ replacing the R-AIPP method. The third subsection presents a summary of the numerical experiments.

We first describe the three different R-AIPP variants considered. While the second variant does not assume knowledge of an upper bound $m$ on the quantity $\underline{m}$ in \eqref{eq:m_lower_def}, the first and third variants do in order to determine their initial stepsize $\lam_0$. More specifically, the first variant, referred to as R-AIPPc, is the R-AIPP method
with initial stepsize chosen to be $\lam_0 =0.9/(2m)$. As opposed
to the two algorithms explained below, which can adaptively change
$\lam_{k}$ between iterations, this algorithm is a constant stepsize
method (see Lemma~\ref{lem1:adaptiveMet} and the paragraph following
it). The second variant, referred to as R-AIPPv1, is the R-AIPP
method with initial stepsize chosen to be $\lam_0 =1$. Since
$\lam_0$ is relatively large in the experiments considered,
$\lambda$ is halved in some of its outer iterations. The third variant,
referred to as R-AIPPv2, is a variant of the R-AIPP method with initial
stepsize chosen to be $\lam_0 =1/(5m)$. This variant modifies
the R-AIPP method by adding conditions that allow the stepsize $\lambda$
to increase between subproblems. More specifically, the R-AIPPv2 method
doubles the value of $\lambda$ at the end of iteration $k$ when:
(a) $\lambda$ has never been halved in step 1 or 2 and (b) the number
of inner iterations performed by the R-ACG algorithm in step~1 is
less than 250. All R-AIPP variants are run with $\theta=4$, a problem--specific
value of $\tau$, and adaptively estimate the constant $\widetilde{M}$ that is used in each iteration of the R-ACG algorithm.

We now make three remarks about the above R-AIPP variants and the AIPP method of \cite{WJRproxmet1}. First, while both the R-AIPPc and AIPP method choose the stepsizes $\{\lambda_k\}$ to be constant, the former method differs from the latter one in that it uses a more relaxed criterion, i.e., \eqref{eq:bd_prox-approx} and \eqref{eq:eps_gsm_bd}, for solving the $k^{\rm th}$ prox subproblem \eqref{eq:penPbRegIntroa}. Moreover, the limited numerical experiments in Appendix~\ref{app:comp_aipp} show that this relaxation drastically improves upon the efficiency of the AIPP method, regardless of the magnitude of the ratio $M/m$. As we believe that this effect would observed in the other problem instances of this section, we choose not to include the AIPP method as part of our suite of benchmark algorithms for the sake of brevity.
%have shown that the R-AIPPc method is the more efficient of the two, and we believe that this superior performance is attributed to the fact that it, being an instance of the GD framework, uses a more relaxed criterion (see the discussion after Proposition~\ref{prop:GIPP1}) to compute the iterate $(\lam_k,z_k,v_k)$. 
Second, the R-AIPPv1 and R-AIPPv2 methods differ from the R-AIPPc method in that they permit the stepsizes $\{\lam_k\}$ to be significantly larger than the constant ones chosen for the R-AIPPc method. As will be observed in the numerical experiments below, this can drastically improve the efficiency of the adaptive stepsize R-AIPP variants.
Third, in view of the descriptions of the R-AIPP variants in the previous paragraph, both the R-AIPPc and R-AIPPv1 methods are instances of the R-AIPP method while the R-AIPPv2 method is not. However, the R-AIPPv2 method is clearly an instance of the GD framework, and hence a similar analysis to the one in Section~\ref{sec:AIPPmet} may be used to establish its ACG iteration complexity. For sake of brevity we omit its analysis in this paper.

% Proposition~\ref{prop:relaxAIPPmethod} of Section~\ref{sec:AIPPmet}. While we believe that a similar analysis can be given for the AIPPv2 method, we omit it here for the sake of brevity.

We now describe the three other nonconvex composite optimization algorithms
considered. The first algorithm is an implementation of the unified
problem-parameter free accelerated gradient (UPFAG) method that is
proposed and analyzed in \cite{Ghadimi2019}. The particular implementation
considered is the UPFAG-fullBB method, which utilizes a Barzilai--Borwein
type stepsize selection strategy and is described in \cite[Section 4]{Ghadimi2019}.
Its input parameters include $(\gamma_{1},\gamma_{2},\gamma_{3})=(0.4,0.4,1.0)$
and $(\delta,\sigma)=(10^{-2},10^{-10})$. The second algorithm is
an implementation of the NC-FISTA method in \cite{liang2019fistatype}.
The particular implementation considered uses input parameters $(\xi,\lam)=(1.05m,0.99/M)$.
The third algorithm is an implementation of the accelerated gradient
(AG) method that is proposed and analyzed in \cite{nonconv_lan16}.
The particular implementation considered is Algorithm 2, which is
described in \cite[Section 2]{nonconv_lan16}. 

Finally, we state some additional details about the numerical experiments.
First, for each linearly--set--constrained problem of the form given
in \eqref{eq:main_lin_prb}, the quadratic penalty method used to
solve it starts with the initial penalty parameter chosen to be $c_{0}=\max\{10^{-10},(1000m-L)/\|A\|^{2}\}$.
Second, each algorithm is run with a time limit of 4000 seconds. If
an algorithm does not terminate with a solution for a particular problem
instance, we do not report any details about its iteration count or
function value at the point of termination and the runtime for that
instance is marked with a {[}{*}{]} symbol. Third, the iterations
listed in the tables this section include backtracking iterations
if a parameter line search method is used as part of the algorithm.
Finally, all algorithms described at the beginning of this section
are implemented in MATLAB 2019a and are run on Linux 64-bit machines
each containing Xeon E5520 processors and at least 8 GB of memory.

\subsection{Unconstrained problems}

This subsection examines the performance of the R-AIPP method as a
nonconvex composite optimization solver for solving problems of the
form given in \eqref{eq:PenProb2Introa}. Given a function pair $(g,h)$
satisfying assumptions (A1)--(A3) with $\phi=g+h$, tolerance $\hat{\rho}>0$,
and an initial point $z_{0}\in\dom h$, each algorithm seeks a pair
$(\hat{z},\hat{v})$ satisfying 
\begin{equation}
\hat{v}\in\nabla g(\hat{z})+\pt h(\hat{z}),\quad\frac{\|\hat{v}\|}{\|\nabla g(z_{0})\|+1}\leq\hat{\rho}.\label{eq:term_unconstr}
\end{equation}
Two problems are considered, namely: (i) the quadratic matrix problem;
and (ii) the support vector machine problem in \cite{Ghadimi2019}. 

All methods that terminated within 4000 seconds converged to the same
objective value, which, for each table in this subsection, is given
in a column labeled $\phi(\hat{z})$. The bold numbers in each of
the aforementioned tables highlight the algorithm that performed the
most efficiently in terms of iteration count or total runtime. 

\subsubsection{Quadratic matrix problem} \label{subsubsec:qm_prb}

Given a pair of dimensions $(l,n)\in\n^{2}$, scalar pair $(\alpha_{1},\alpha_{2})\in\r_{++}^{2}$,
linear operators ${\cal B}:S_{+}^{n}\mapsto\r^{n}$ and ${\cal C}:S_{+}^{n}\mapsto\r^{l}$
defined by
\[
\left[{\cal B}(z)\right]_{j}=\left\langle B_{j},z\right\rangle _{F},\quad\left[{\cal C}(z)\right]_{i}=\left\langle C_{i},z\right\rangle _{F},
\]
for matrices $\{B_{j}\}_{j=1}^{n},\{C_{i}\}_{i=1}^{l}\subseteq\r^{n\times n}$
, positive diagonal matrix $D\in\r^{n\times n}$, and vector $d\in\r^{l}$,
this sub--subsection considers the following quadratic matrix (QM)
problem: 
\begin{align*}
\min_{z}\  & \frac{\alpha_{1}}{2}\|{\cal C}(z)-d\|^{2}-\frac{\alpha_{2}}{2}\|D{\cal B}(z)\|^{2}\\
\text{s.t.}\  & z\in P_{n},
\end{align*}
where $P_{n}=\{z\in S_{+}^{n}:\trc z=1\}$ denotes the $n$--dimensional
spectraplex.

We now describe the experiment parameters for the instances considered.
First, the dimensions were set to be $(l,n)=(50,200)$ and only 2.5\%
of the entries of the submatrices $B_{j}$ and $C_{i}$ being nonzero.
Second, the entries of $B_{j},C_{i}$, and $d$ (resp., $D$) are
generated by sampling from the uniform distribution ${\cal U}[0,1]$
(resp., ${\cal U}[1,1000]$). Third, the initial starting point is
$z_{0}=I_{n}/n$, where $I_{n}$ is the $n$-dimensional identity
matrix. Fourth, with respect to the termination criterion \eqref{eq:term_unconstr},
the inputs, for every $z\in S_{+}^{n}$, are 
\[
g(z)=\frac{\alpha_{1}}{2}\|{\cal C}(z)-d\|^{2}-\frac{\alpha_{2}}{2}\|D{\cal B}(z)\|^{2},\quad h(z)=\delta_{P_{n}}(z),\quad\hat{\rho}=10^{-7}.
\]
Fifth, the R-AIPP variants used a parameter value of $\tau=10000$.
Finally, each problem instance considered is based on a specific curvature
pair $(m,M)\in\r_{++}^{2}$ for which the scalar pair $(\alpha_{1},\alpha_{2})\in\r_{++}^{2}$
is selected so that $M=\lambda_{\max}(\nabla^{2}g)$ and $-m=\lambda_{\min}(\nabla^{2}g)$.

We now present the numerical tables for this set of problem instances.
We start with instances in which $m$ is fixed.

\begin{table}[H]
    \begin{centering}
        \makebox[\textwidth][c]{%
            \begin{tabular}{|>{\centering}m{0.7cm}>{\centering}m{0.7cm}|>{\centering}p{1.8cm}|>{\centering}p{1.7cm}>{\centering}p{1.7cm}>{\centering}p{1.7cm}>{\centering}p{1.7cm}>{\centering}p{1.7cm}>{\centering}p{1.7cm}|}
                \hline 
                \multirow{2}{0.7cm}{\centering{}{\footnotesize{}$M$}} & \multirow{2}{0.7cm}{\centering{}{\footnotesize{}$m$}} & \multirow{2}{1.8cm}{\centering{}{\footnotesize{}$\phi(\hat{z})$}} & \multicolumn{6}{c|}{{\small{}Iteration Count}}\tabularnewline
                &  &  & {\footnotesize{}UPFAG} & {\footnotesize{}NC-FISTA} & {\footnotesize{}AG} & {\footnotesize{}R-AIPPc} & {\footnotesize{}R-AIPPv1} & {\footnotesize{}R-AIPPv2}\tabularnewline
                \hline 
                {\footnotesize{}$10^{2}$} & {\footnotesize{}$10^{0}$} & {\footnotesize{}--1.74E--02} & {\footnotesize{}3892} & {\footnotesize{}2045} & {\footnotesize{}8670} & {\footnotesize{}8266} & {\footnotesize{}7627} & \textbf{\footnotesize{}1093}\tabularnewline
                {\footnotesize{}$10^{4}$} & {\footnotesize{}$10^{0}$} & {\footnotesize{}3.67E--01} & {\footnotesize{}9809} & {\footnotesize{}8642} & {\footnotesize{}7064} & {\footnotesize{}3250} & {\footnotesize{}3691} & \textbf{\footnotesize{}1185}\tabularnewline
                {\footnotesize{}$10^{6}$} & {\footnotesize{}$10^{0}$} & {\footnotesize{}3.84E+01} & {\footnotesize{}23388} & {\footnotesize{}11861} & {\footnotesize{}7039} & {\footnotesize{}1270} & {\footnotesize{}1268} & \textbf{\footnotesize{}1174}\tabularnewline
                \hline 
        \end{tabular}}
        \par\end{centering}
    \caption{Iteration counts for QM problems with fixed $m$.}\label{tab:high_ratio_1}
\end{table}

\begin{table}[H]
    \begin{centering}
        \makebox[\textwidth][c]{%
            \begin{tabular}{|>{\centering}m{0.7cm}>{\centering}m{0.7cm}|>{\centering}p{1.8cm}|>{\centering}p{1.7cm}>{\centering}p{1.7cm}>{\centering}p{1.7cm}>{\centering}p{1.7cm}>{\centering}p{1.7cm}>{\centering}p{1.7cm}|}
                \hline 
                \multirow{2}{0.7cm}{\centering{}{\footnotesize{}$M$}} & \multirow{2}{0.7cm}{\centering{}{\footnotesize{}$m$}} & \multirow{2}{1.8cm}{\centering{}{\footnotesize{}$\phi(\hat{z})$}} & \multicolumn{6}{c|}{{\small{}Runtime (seconds)}}\tabularnewline
                &  &  & {\footnotesize{}UPFAG} & {\footnotesize{}NC-FISTA} & {\footnotesize{}AG} & {\footnotesize{}R-AIPPc} & {\footnotesize{}R-AIPPv1} & {\footnotesize{}R-AIPPv2}\tabularnewline
                \hline 
                {\footnotesize{}$10^{2}$} & {\footnotesize{}$10^{0}$} & {\footnotesize{}--1.74E--02} & {\footnotesize{}511.91} & {\footnotesize{}151.58} & {\footnotesize{}880.19} & {\footnotesize{}935.50} & {\footnotesize{}981.49} & \textbf{\footnotesize{}119.87}\tabularnewline
                {\footnotesize{}$10^{4}$} & {\footnotesize{}$10^{0}$} & {\footnotesize{}3.67E--01} & {\footnotesize{}1304.09} & {\footnotesize{}683.95} & {\footnotesize{}687.16} & {\footnotesize{}287.37} & {\footnotesize{}334.41} & \textbf{\footnotesize{}106.54}\tabularnewline
                {\footnotesize{}$10^{6}$} & {\footnotesize{}$10^{0}$} & {\footnotesize{}3.84E+01} & {\footnotesize{}2804.38} & {\footnotesize{}774.74} & {\footnotesize{}615.91} & {\footnotesize{}100.27} & {\footnotesize{}102.29} & \textbf{\footnotesize{}94.16}\tabularnewline
                \hline 
        \end{tabular}}
        \par\end{centering}
    \caption{Runtimes for QM problems with fixed $m$.}\label{tab:high_ratio_2}
\end{table}

\noindent We now present instances where $m=M$.

\begin{table}[H]
    \begin{centering}
        \makebox[\textwidth][c]{%
            \begin{tabular}{|>{\centering}m{0.7cm}>{\centering}m{0.7cm}|>{\centering}p{1.8cm}|>{\centering}p{1.7cm}>{\centering}p{1.7cm}>{\centering}p{1.7cm}>{\centering}p{1.7cm}>{\centering}p{1.7cm}>{\centering}p{1.7cm}|}
                \hline 
                \multirow{2}{0.7cm}{\centering{}{\footnotesize{}$M$}} & \multirow{2}{0.7cm}{\centering{}{\footnotesize{}$m$}} & \multirow{2}{1.8cm}{\centering{}{\footnotesize{}$\phi(\hat{z})$}} & \multicolumn{6}{c|}{{\small{}Iteration Count}}\tabularnewline
                &  &  & {\footnotesize{}UPFAG} & {\footnotesize{}NC-FISTA} & {\footnotesize{}AG} & {\footnotesize{}R-AIPPc} & {\footnotesize{}R-AIPPv1} & {\footnotesize{}R-AIPPv2}\tabularnewline
                \hline 
                {\footnotesize{}$10^{2}$} & {\footnotesize{}$10^{2}$} & {\footnotesize{}-2.06E+01} & {\footnotesize{}18} & {\footnotesize{}38} & {\footnotesize{}79} & {\footnotesize{}161} & \textbf{\footnotesize{}10} & {\footnotesize{}20}\tabularnewline
                {\footnotesize{}$10^{4}$} & {\footnotesize{}$10^{4}$} & {\footnotesize{}-2.06E+03} & {\footnotesize{}19} & {\footnotesize{}39} & {\footnotesize{}80} & {\footnotesize{}217} & \textbf{\footnotesize{}7} & {\footnotesize{}21}\tabularnewline
                {\footnotesize{}$10^{6}$} & {\footnotesize{}$10^{6}$} & {\footnotesize{}-2.06E+05} & {\footnotesize{}19} & {\footnotesize{}39} & {\footnotesize{}80} & {\footnotesize{}175} & \textbf{\footnotesize{}8} & {\footnotesize{}20}\tabularnewline
                \hline 
        \end{tabular}}
        \par\end{centering}
    \caption{Iteration counts for QM problems with $m=M$.}
    
    \label{tab:low_ratio_1}
\end{table}

\begin{table}[H]
    \begin{centering}
        \makebox[\textwidth][c]{%
            \begin{tabular}{|>{\centering}m{0.7cm}>{\centering}m{0.7cm}|>{\centering}p{1.8cm}|>{\centering}p{1.7cm}>{\centering}p{1.7cm}>{\centering}p{1.7cm}>{\centering}p{1.7cm}>{\centering}p{1.7cm}>{\centering}p{1.7cm}|}
                \hline 
                \multirow{2}{0.7cm}{\centering{}{\footnotesize{}$M$}} & \multirow{2}{0.7cm}{\centering{}{\footnotesize{}$m$}} & \multirow{2}{1.8cm}{\centering{}{\footnotesize{}$\phi(\hat{z})$}} & \multicolumn{6}{c|}{{\small{}Runtime (seconds)}}\tabularnewline
                &  &  & {\footnotesize{}UPFAG} & {\footnotesize{}NC-FISTA} & {\footnotesize{}AG} & {\footnotesize{}R-AIPPc} & {\footnotesize{}R-AIPPv1} & {\footnotesize{}R-AIPPv2}\tabularnewline
                \hline 
                {\footnotesize{}$10^{2}$} & {\footnotesize{}$10^{2}$} & {\footnotesize{}-2.06E+01} & {\footnotesize{}1.73} & {\footnotesize{}1.71} & {\footnotesize{}4.80} & {\footnotesize{}16.07} & \textbf{\footnotesize{}1.23} & {\footnotesize{}2.20}\tabularnewline
                {\footnotesize{}$10^{4}$} & {\footnotesize{}$10^{4}$} & {\footnotesize{}-2.06E+03} & {\footnotesize{}1.68} & {\footnotesize{}1.91} & {\footnotesize{}4.89} & {\footnotesize{}19.67} & \textbf{\footnotesize{}0.70} & {\footnotesize{}2.35}\tabularnewline
                {\footnotesize{}$10^{6}$} & {\footnotesize{}$10^{6}$} & {\footnotesize{}-2.06E+05} & {\footnotesize{}2.06} & {\footnotesize{}2.06} & {\footnotesize{}4.73} & {\footnotesize{}16.32} & \textbf{\footnotesize{}0.57} & {\footnotesize{}2.08}\tabularnewline
                \hline 
        \end{tabular}}
        \par\end{centering}
    \caption{Runtimes for QM problems with $m=M$.}
    \label{tab:low_ratio_2}
\end{table}

\subsubsection{Support vector machine problem}

Given a pair of dimensions $(n,k)\in\n^{2}$, matrix $U\in\r^{n\times k},$
and vector $v\in\{-1,+1\}^{n},$ this sub--subsection considers the
following (sigmoid) support vector machine (SVM) problem
\[
\min_{z}\ \frac{1}{k}\sum_{i=1}^{k}\left[1-\tanh\left(v_{i}\left\langle u_{i},z\right\rangle \right)\right]+\frac{1}{2k}\|z\|^{2},
\]
where $u_{i}$ denotes the $i^{\rm th}$ column of $U$.

We now describe the experiment parameters for the instances considered.
First, the entries of $U$ are generated by sampling from the uniform
distribution ${\cal U}[0,1]$, with only 5\% of the entries being
nonzero, and $v=\mathrm{sgn}(U^{T}x)$ where the entries of $x$ are
sampled from the uniform distribution over the $k$--dimensional
ball centered at 0 with radius 50. Second, the initial starting point
is $z_{0}=0$. Third, the curvature parameters for each problem instance
are $m=M=(4\sqrt{3}\|U\|_{F}^{2})/(9k)+1/k.$ Fourth, with respect
to the termination criterion \eqref{eq:term_unconstr}, the inputs,
for every $z\in\r^{n}$, are 
\[
g(z)=\frac{1}{k}\sum_{i=1}^{k}\left[1-\tanh\left(v_{i}\left\langle u_{i},z\right\rangle \right)\right]+\frac{1}{2k}\|z\|^{2},\quad h(z)=0,\quad\hat{\rho}=10^{-3}.
\]
Fifth, the R-AIPP variants used a parameter value of $\tau=5000$.
Finally, each problem instance considered is based on a specific dimension
pair $(n,k)\in\n^{2}$.

We now present the numerical tables for this set of problem instances.

\begin{table}[H]
    \begin{centering}
        \makebox[\textwidth][c]{%
            \begin{tabular}{|>{\centering}m{0.7cm}>{\centering}m{0.7cm}|>{\centering}p{1.8cm}|>{\centering}p{1.7cm}>{\centering}p{1.7cm}>{\centering}p{1.7cm}>{\centering}p{1.7cm}>{\centering}p{1.7cm}>{\centering}p{1.7cm}|}
                \hline 
                \multirow{2}{0.7cm}{\centering{}{\footnotesize{}$n$}} & \multirow{2}{0.7cm}{\centering{}{\footnotesize{}$k$}} & \multirow{2}{1.8cm}{\centering{}{\footnotesize{}$\phi(\hat{z})$}} & \multicolumn{6}{c|}{{\small{}Iteration Count}}\tabularnewline
                &  &  & {\footnotesize{}UPFAG} & {\footnotesize{}NC-FISTA} & {\footnotesize{}AG} & {\footnotesize{}R-AIPPc} & {\footnotesize{}R-AIPPv1} & {\footnotesize{}R-AIPPv2}\tabularnewline
                \hline 
                {\footnotesize{}1000} & {\footnotesize{}500} & {\footnotesize{}2.37E--01} & {\footnotesize{}82} & {\footnotesize{}3025} & {\footnotesize{}783} & {\footnotesize{}8234} & {\footnotesize{}889} & \textbf{\footnotesize{}57}\tabularnewline
                {\footnotesize{}2000} & {\footnotesize{}1000} & {\footnotesize{}1.61E--01} & {\footnotesize{}197} & {\footnotesize{}8361} & {\footnotesize{}1192} & {\footnotesize{}22706} & {\footnotesize{}1227} & \textbf{\footnotesize{}85}\tabularnewline
                {\footnotesize{}4000} & {\footnotesize{}2000} & {\footnotesize{}1.05E--01} & {\footnotesize{}1128} & {\footnotesize{}-} & {\footnotesize{}1347} & {\footnotesize{}-} & {\footnotesize{}1651} & \textbf{\footnotesize{}97}\tabularnewline
                {\footnotesize{}8000} & {\footnotesize{}4000} & {\footnotesize{}6.67E--02} & {\footnotesize{}372} & {\footnotesize{}-} & {\footnotesize{}1647} & {\footnotesize{}-} & {\footnotesize{}-} & \textbf{\footnotesize{}148}\tabularnewline
                \hline 
        \end{tabular}}
        \par\end{centering}
    \caption{Iteration counts for SVM problems.}
\end{table}

\begin{table}[H]
    \begin{centering}
        \makebox[\textwidth][c]{%
            \begin{tabular}{|>{\centering}m{0.7cm}>{\centering}m{0.7cm}|>{\centering}p{1.8cm}|>{\centering}p{1.7cm}>{\centering}p{1.7cm}>{\centering}p{1.7cm}>{\centering}p{1.7cm}>{\centering}p{1.7cm}>{\centering}p{1.7cm}|}
                \hline 
                \multirow{2}{0.7cm}{\centering{}{\footnotesize{}$n$}} & \multirow{2}{0.7cm}{\centering{}{\footnotesize{}$k$}} & \multirow{2}{1.8cm}{\centering{}{\footnotesize{}$\phi(\hat{z})$}} & \multicolumn{6}{c|}{{\small{}Runtime (seconds)}}\tabularnewline
                &  &  & {\footnotesize{}UPFAG} & {\footnotesize{}NC-FISTA} & {\footnotesize{}AG} & {\footnotesize{}R-AIPPc} & {\footnotesize{}R-AIPPv1} & {\footnotesize{}R-AIPPv2}\tabularnewline
                \hline 
                {\footnotesize{}1000} & {\footnotesize{}500} & {\footnotesize{}2.37E--01} & {\footnotesize{}3.58} & {\footnotesize{}49.00} & {\footnotesize{}12.80} & {\footnotesize{}389.39} & {\footnotesize{}35.21} & \textbf{\footnotesize{}1.74}\tabularnewline
                {\footnotesize{}2000} & {\footnotesize{}1000} & {\footnotesize{}1.61E--01} & {\footnotesize{}29.05} & {\footnotesize{}473.67} & {\footnotesize{}65.79} & {\footnotesize{}3626.51} & {\footnotesize{}164.56} & \textbf{\footnotesize{}7.73}\tabularnewline
                {\footnotesize{}4000} & {\footnotesize{}2000} & {\footnotesize{}1.05E--01} & {\footnotesize{}1076.09} & {\footnotesize{}4000.00{*}} & {\footnotesize{}437.80} & {\footnotesize{}4000.00{*}} & {\footnotesize{}1059.98} & \textbf{\footnotesize{}47.75}\tabularnewline
                {\footnotesize{}8000} & {\footnotesize{}4000} & {\footnotesize{}6.67E--02} & {\footnotesize{}1118.84} & {\footnotesize{}4000.00{*}} & {\footnotesize{}1975.18} & {\footnotesize{}4000.00{*}} & {\footnotesize{}4000.00{*}} & \textbf{\footnotesize{}177.03}\tabularnewline
                \hline 
        \end{tabular}}
        \par\end{centering}
    \caption{Runtimes for SVM problems.}
\end{table}

\subsection{Linearly constrained problems}

This subsection examines the performance of the R-QP-AIPP method as
a nonconvex linearly--set--constrained composite optimization solver
for solving problems of the form given in \eqref{eq:main_lin_prb}.
Given a linear operator $A$, convex set $S$, function pair $(f,h)$
satisfying assumptions (C1)--(C3), tolerance pair $(\hat{\rho},\hat{\eta})\in\r_{++}^{2}$,
and an initial point $z_{0}\in\dom h$, each algorithm seeks a triple
$([\hat{z},\hat{s}],\hat{p},\hat{v})$ satisfying
\begin{gather}
\begin{gathered}\hat{v}\in\nabla f(\hat{z})+\pt h(\hat{z})+A^{*}\hat{p},\quad\frac{\|\hat{v}\|}{\|\nabla f(z_{0})\|+1}\leq\hat{\rho},\\
\|A\hat{z}-\hat{s}\|\le\hat{\eta},\quad\hat{p}\in N_{S}(\hat{s}).
\end{gathered}
\label{eq:term_lin_constr}
\end{gather}
Three problems are considered, namely: (i) the linearly--constrained
quadratic matrix problem; (ii) the sparse principal component analysis
problem in \cite{NIPS2014_5615}; and (iii) the bounded matrix completion
problem in \cite{yao2017efficient}. 

The bold numbers in each of the tables in this subsection highlight
the algorithm that performed the most efficiently in terms of iteration
count or total runtime. 

\subsubsection{Linearly--constrained quadratic matrix problem}

Given a pair of dimensions $(l,n)\in\n^{2}$, scalar pair $(\alpha_{1},\alpha_{2})\in\r_{++}^{2}$,
linear operators ${\cal A}:S_{+}^{n}\mapsto\r^{l}$ , ${\cal B}:S_{+}^{n}\mapsto\r^{n}$,
and ${\cal C}:S_{+}^{n}\mapsto\r^{l}$ defined by
\[
\left[{\cal A}(z)\right]_{i}=\left\langle A_{i},z\right\rangle _{F},\quad\left[{\cal B}(z)\right]_{j}=\left\langle B_{j},z\right\rangle _{F},\quad\left[{\cal C}(z)\right]_{i}=\left\langle C_{i},z\right\rangle _{F},
\]
for matrices $\{A_{i}\}_{i=1}^{l},\{B_{j}\}_{j=1}^{n},\{C_{i}\}_{i=1}^{l}\subseteq\r^{n\times n}$,
positive diagonal matrix $D\in\r^{n\times n}$, and vector pair $(b,d)\in\r^{l}\times\r^{l}$,
this sub--subsection considers the following linearly--constrained
quadratic matrix (LCQM) problem:
\begin{align*}
\min_{z}\  & \frac{\alpha_{1}}{2}\|{\cal C}(z)-d\|^{2}-\frac{\alpha_{2}}{2}\|D{\cal B}(z)\|^{2}\\
\text{s.t.}\  & {\cal A}(z)\in\{b\},\quad z\in P_{n},
\end{align*}
where $P_{n}=\{z\in S_{+}^{n}:\trc z=1\}$ denotes the $n$--dimensional
spectraplex.

We now describe the experiment parameters for the instances considered.
First, the dimensions were set to be $(l,n)=(50,200)$ and only 1.0\%
of the entries of the submatrices $A_{i},B_{j},$ and $C_{i}$ being
nonzero. Second, the entries of $A_{i},B_{j},C_{i},b$, and $d$ (resp.,
$D$) were generated by sampling from the uniform distribution ${\cal U}[0,1]$
(resp., ${\cal U}[1,1000]$). Third, the initial starting point $z_{0}$
was chosen to be a random point in $S_{+}^{n}$. More specifically,
three unit vectors $\nu_{1},\nu_{2},\nu_{3}\in\r^{n}$ and three scalars
$e_{1},e_{2},e_{2}\in\r_{+}$ are first generated by sampling vectors
$\widetilde{\nu}_{i}\sim{\cal U}^{n}[0,1]$ and scalars $\widetilde{d}_{i}\sim{\cal U}[0,1]$
and setting $\nu_{i}=\widetilde{\nu}_{i}/\|\widetilde{\nu}_{i}\|$ and $e_{i}=\widetilde{e}_{i}/(\sum_{j=1}^{3}\widetilde{e}_{i})$
for $i=1,2,3$. The initial iterate for the first subproblem is then
set to $z_{0}=\sum_{i=1}^{3}e_{i}\nu_{i}\nu_{i}^{T}$. Fourth, with
respect to the termination criterion \eqref{eq:term_unconstr}, the
inputs, for every $z\in S_{+}^{n}$, are 
\begin{gather*}
f(z)=\frac{\alpha_{1}}{2}\|{\cal C}(z)-d\|^{2}-\frac{\alpha_{2}}{2}\|D{\cal B}(z)\|^{2},\quad h(z)=\delta_{P_{n}}(z),\\
A(z)={\cal A}(z),\quad S=\{b\},\quad\hat{\rho}=10^{-3},\quad\hat{\eta}=10^{-3}.
\end{gather*}
Fifth, the R-AIPP variants used a parameter value of $\tau=5000$.
Finally, each problem instance considered is based on a specific curvature
pair $(m,M)\in\r_{++}^{2}$ for which the scalar pair $(\alpha_{1},\alpha_{2})\in\r_{++}^{2}$
is selected so that $M=\lambda_{\max}(\nabla^{2}f)$ and $-m=\lambda_{\min}(\nabla^{2}f)$.

We now present the numerical tables for this set of problem instances.

\begin{table}[H]
    \begin{centering}
        \makebox[\textwidth][c]{%
            \begin{tabular}{|>{\centering}m{0.7cm}>{\centering}m{0.7cm}|>{\centering}p{1.7cm}>{\centering}p{1.7cm}>{\centering}p{1.7cm}>{\centering}p{1.7cm}>{\centering}p{1.7cm}>{\centering}p{1.7cm}|}
                \hline 
                \multirow{2}{0.7cm}{\centering{}{\footnotesize{}$L$}} & \multirow{2}{0.7cm}{\centering{}{\footnotesize{}$m$}} & \multicolumn{6}{c|}{{\small{}Iteration Count}}\tabularnewline
                &  & {\footnotesize{}UPFAG} & {\footnotesize{}NC-FISTA} & {\footnotesize{}AG} & {\footnotesize{}R-AIPPc} & {\footnotesize{}R-AIPPv1} & {\footnotesize{}R-AIPPv2}\tabularnewline
                \hline 
                {\footnotesize{}$10^{1}$} & {\footnotesize{}$10^{0}$} & {\footnotesize{}2148} & {\footnotesize{}12758} & {\footnotesize{}8739} & {\footnotesize{}1797} & {\footnotesize{}1675} & \textbf{\footnotesize{}998}\tabularnewline
                {\footnotesize{}$10^{2}$} & {\footnotesize{}$10^{0}$} & {\footnotesize{}1615} & {\footnotesize{}8957} & {\footnotesize{}5253} & {\footnotesize{}1206} & \textbf{\footnotesize{}1103} & {\footnotesize{}1153}\tabularnewline
                {\footnotesize{}$10^{3}$} & {\footnotesize{}$10^{0}$} & {\footnotesize{}3967} & {\footnotesize{}26383} & {\footnotesize{}5926} & {\footnotesize{}1570} & \textbf{\footnotesize{}1489} & {\footnotesize{}1555}\tabularnewline
                \hline 
        \end{tabular}}
        \par\end{centering}
    \caption{Iteration counts for LCQM problems.}
\end{table}

\begin{table}[H]
    \begin{centering}
        \makebox[\textwidth][c]{%
            \begin{tabular}{|>{\centering}m{0.7cm}>{\centering}m{0.7cm}|>{\centering}p{1.7cm}>{\centering}p{1.7cm}>{\centering}p{1.7cm}>{\centering}p{1.7cm}>{\centering}p{1.7cm}>{\centering}p{1.7cm}|}
                \hline 
                \multirow{2}{0.7cm}{\centering{}{\footnotesize{}$L$}} & \multirow{2}{0.7cm}{\centering{}{\footnotesize{}$m$}} & \multicolumn{6}{c|}{{\small{}Runtime (seconds)}}\tabularnewline
                &  & {\footnotesize{}UPFAG} & {\footnotesize{}NC-FISTA} & {\footnotesize{}AG} & {\footnotesize{}R-AIPPc} & {\footnotesize{}R-AIPPv1} & {\footnotesize{}R-AIPPv2}\tabularnewline
                \hline 
                {\footnotesize{}$10^{1}$} & {\footnotesize{}$10^{0}$} & {\footnotesize{}274.13} & {\footnotesize{}958.47} & {\footnotesize{}883.50} & {\footnotesize{}205.48} & {\footnotesize{}192.35} & \textbf{\footnotesize{}103.60}\tabularnewline
                {\footnotesize{}$10^{2}$} & {\footnotesize{}$10^{0}$} & {\footnotesize{}218.05} & {\footnotesize{}684.10} & {\footnotesize{}531.88} & {\footnotesize{}124.45} & {\footnotesize{}117.54} & \textbf{\footnotesize{}117.32}\tabularnewline
                {\footnotesize{}$10^{3}$} & {\footnotesize{}$10^{0}$} & {\footnotesize{}481.51} & {\footnotesize{}1997.85} & {\footnotesize{}615.14} & {\footnotesize{}165.38} & \textbf{\footnotesize{}156.69} & {\footnotesize{}164.04}\tabularnewline
                \hline 
        \end{tabular}}
        \par\end{centering}
    \caption{Runtimes for LCQM problems.}
\end{table}

\subsubsection{Sparse principal component analysis problem}

Given integer $k$, positive scalar pair $(\nu,b)\in\r_{++}^{2}$,
and matrix $\Sigma\in S_{+}^{n}$, this sub--subsection considers
the following sparse principal component analysis (PCA) problem:
\begin{align*}
\min_{\Pi,\Phi}\  & \left\langle \Sigma,\Pi\right\rangle _{F}+\sum_{i,j=1}^{n}q_{\nu}(\Phi_{ij})+\nu\sum_{i,j=1}^{n}|\Phi_{ij}|\\
\text{s.t.}\  & \Pi-\Phi=0,\quad(\Pi,\Phi)\in{\cal F}^{k}\times\r^{n\times n}
\end{align*}
where ${\cal F}^{k}=\{z\in S_{+}^{n}:0\preceq z\preceq I,\trc M=k\}$
denotes the $k$--Fantope and $q_{\nu}$ is the minimax concave penalty
(MCP) function given by
\[
q_{\nu}(t):=\begin{cases}
-t^{2}/(2b), & \text{if }|t|\leq b\nu,\\
b\nu^{2}/2-\nu|t|, & \text{if }|t|>b\nu,
\end{cases}\quad\forall t\in\r.
\]

We now describe the experiment parameters for the instances considered.
First, the scalar parameters are chosen to be $(\nu,b)=(100,100,0.1)$.
Second, the matrix $\Sigma$ is generated according to an eigenvalue
decomposition $\Sigma=P\Lambda P^{T}$, based on a parameter pair
$(s,k)$, where $k$ is as in the problem description and $s$ is
a positive integer. In particular, we choose $\Lambda=(100,1,...,1)$,
the first column of $P$ to be a sparse vector whose first $s$ entries
are $1/\sqrt{s}$, and the other entries of $P$ to be sampled randomly
from the standard Gaussian distribution. Third, the initial starting
point is $(\Pi_{0},\Phi_{0})=(D_{k},0)$ where $D_{k}$ is a diagonal
matrix whose first $k$ entries are 1 and whose remaining entries
are 0. Fourth, the curvature parameters for each problem instance
are $m=M=1/b.$ Fifth, with respect to the termination criterion \eqref{eq:term_unconstr},
the inputs, for every $(\Pi,\Phi)\in S_{+}^{n}\times\r^{n\times n}$,
are
\begin{gather*}
f(\Pi,\Phi)=\left\langle \Sigma,\Pi\right\rangle _{F}+\sum_{i,j=1}^{n}q_{\nu}(\Phi_{ij}),\quad h(\Pi,\Phi)=\delta_{{\cal F}^{k}}(\Pi)+\nu\sum_{i,j=1}^{n}|\Phi_{ij}|,\\
A(\Pi,\Phi):=\Pi-\Phi,\quad S=\{0\},\quad\hat{\eta}=10^{-3},\quad\hat{\rho}=10^{-6}.
\end{gather*}
Sixth, the R-AIPP variants used a parameter value of $\tau=100000$.
Finally, each problem instance considered is based on a specific parameter
pair $(s,k)\in\n^{2}$ where $s$ is part of the process of generating
$\Sigma$ (see the second description above).

We now present the numerical tables for this set of problem instances.

\begin{table}[H]
    \begin{centering}
        \makebox[\textwidth][c]{%
            \begin{tabular}{|>{\centering}m{0.7cm}>{\centering}m{0.7cm}|>{\centering}p{1.7cm}>{\centering}p{1.7cm}>{\centering}p{1.7cm}>{\centering}p{1.7cm}>{\centering}p{1.7cm}>{\centering}p{1.7cm}|}
                \hline 
                \multirow{2}{0.7cm}{\centering{}{\footnotesize{}$s$}} & \multirow{2}{0.7cm}{\centering{}{\footnotesize{}$k$}} & \multicolumn{6}{c|}{{\small{}Iteration Count}}\tabularnewline
                &  & {\footnotesize{}UPFAG} & {\footnotesize{}NC-FISTA} & {\footnotesize{}AG} & {\footnotesize{}R-AIPPc} & {\footnotesize{}R-AIPPv1} & {\footnotesize{}R-AIPPv2}\tabularnewline
                \hline 
                {\footnotesize{}5} & {\footnotesize{}1} & {\footnotesize{}-} & {\footnotesize{}21979} & {\footnotesize{}34584} & \textbf{\footnotesize{}4511} & {\footnotesize{}5735} & {\footnotesize{}6071}\tabularnewline
                {\footnotesize{}10} & {\footnotesize{}1} & {\footnotesize{}-} & {\footnotesize{}23574} & {\footnotesize{}34712} & \textbf{\footnotesize{}4954} & {\footnotesize{}5960} & {\footnotesize{}5745}\tabularnewline
                {\footnotesize{}15} & {\footnotesize{}1} & {\footnotesize{}-} & {\footnotesize{}27944} & {\footnotesize{}32560} & \textbf{\footnotesize{}5197} & {\footnotesize{}5867} & {\footnotesize{}5822}\tabularnewline
                \hline 
        \end{tabular}}
        \par\end{centering}
    \caption{Iteration counts for sparse PCA problems.}
\end{table}

\begin{table}[H]
    \begin{centering}
        \makebox[\textwidth][c]{%
            \begin{tabular}{|>{\centering}m{0.7cm}>{\centering}m{0.7cm}|>{\centering}p{1.7cm}>{\centering}p{1.7cm}>{\centering}p{1.7cm}>{\centering}p{1.7cm}>{\centering}p{1.7cm}>{\centering}p{1.7cm}|}
                \hline 
                \multirow{2}{0.7cm}{\centering{}{\footnotesize{}$s$}} & \multirow{2}{0.7cm}{\centering{}{\footnotesize{}$k$}} & \multicolumn{6}{c|}{{\small{}Runtime (seconds)}}\tabularnewline
                &  & {\footnotesize{}UPFAG} & {\footnotesize{}NC-FISTA} & {\footnotesize{}AG} & {\footnotesize{}R-AIPPc} & {\footnotesize{}R-AIPPv1} & {\footnotesize{}R-AIPPv2}\tabularnewline
                \hline 
                {\footnotesize{}5} & {\footnotesize{}1} & {\footnotesize{}4000.00{*}} & {\footnotesize{}142.11} & {\footnotesize{}349.87} & \textbf{\footnotesize{}67.32} & {\footnotesize{}83.23} & {\footnotesize{}87.99}\tabularnewline
                {\footnotesize{}10} & {\footnotesize{}1} & {\footnotesize{}4000.00{*}} & {\footnotesize{}153.18} & {\footnotesize{}353.59} & \textbf{\footnotesize{}72.72} & {\footnotesize{}86.98} & {\footnotesize{}83.67}\tabularnewline
                {\footnotesize{}15} & {\footnotesize{}1} & {\footnotesize{}4000.00{*}} & {\footnotesize{}180.27} & {\footnotesize{}328.69} & \textbf{\footnotesize{}75.37} & {\footnotesize{}85.56} & {\footnotesize{}84.55}\tabularnewline
                \hline 
        \end{tabular}}
        \par\end{centering}
    \caption{Runtimes for sparse PCA problems.}
\end{table}

\subsubsection{Bounded matrix completion problem\label{subsec:bmc}}

Given a dimension pair $(p,q)\in\n^{2}$, positive scalar triple $(\beta,\mu,\theta)\in\r_{++}^{3}$,
scalar pair $(u,l)\in\r^{2}$, matrix $A\in\r^{p\times q}$, and indices
$\Omega$, this sub--subsection considers the following bounded matrix
completion (BMC) problem: 
\begin{align*}
\min_{X}\  & \frac{1}{2}\|P_{\Omega}(X-A)\|^{2}+\mu\sum_{i=1}^{\min\{p,q\}}\left[\kappa(\sigma_{i}(X))-\kappa_{0}\sigma_{i}(X)\right]+\bar{\mu}\|X\|_{*}\\
\text{s.t.}\  & l\leq X_{ij}\leq u\quad\forall(i,j)\in\{1,...,p\}\times\{1,...,q\},
\end{align*}
where $\|\cdot\|_{*}$ denotes the nuclear norm, the function $P_{\Omega}$
is the linear operator that zeros out any entry not in $\Omega$,
the function $\sigma_{i}(X)$ denotes the $i^{\rm th}$ largest singular
value of $X$, and 
\[
\kappa_{0}:=\frac{\beta}{\theta},\quad\bar{\mu}:=\mu\kappa_{0},\quad\kappa(t):=\beta\log\left(1+\frac{|t|}{\theta}\right)\quad\forall t\in\r.
\]

We now describe the experiment parameters for the instances considered.
First, the matrix $A$ is the user--movie ratings data matrix of
the MovieLens 100K dataset\footnote{See the MovieLens 100K dataset containing 610 users and 9724 movies,
    which is found in \href{https://grouplens.org/datasets/movielens/}{https://grouplens.org/datasets/movielens/}.}, the index set $\Omega$ is the set of nonzero entries in $A$, and
the dimension pair is set to be $(p,q)=(610,9724)$. Second, the initial
starting point was chosen to be $X_{0}=0$. Third, the curvature parameters
for each problem instance are $m=2\beta\mu/\theta^{2}$ and $M=\max\left\{ 1,m\right\} $
and the bounds are set to $(l,u)=(0,5)$. Fourth, with respect to
the termination criterion \eqref{eq:term_unconstr}, the inputs, for
every $X\in\r^{n\times n}$, are 
\begin{gather*}
f(X)=\frac{1}{2}\|P_{\Omega}(X-A)\|^{2}+\mu\sum_{i=1}^{\min\{p,q\}}\left[\kappa(\sigma_{i}(X))-\kappa_{0}\sigma_{i}(X)\right],\quad h(X)=\bar{\mu}\|X\|_{*},\\
A(X)=X,\quad S=\left\{ Z\in\r^{p\times q}:l\leq Z_{ij}\leq u,\:(i,j)\in\{1,...,p\}\times\{1,...,q\}\right\} ,\\
\hat{\eta}=10^{-2},\quad\hat{\rho}=5\times10^{-2}.
\end{gather*}
Fifth, the R-AIPP variants used a parameter value of $\tau=1000$.
Finally, each problem instance considered is based on a specific parameter
triple $(\beta,\mu,\theta)\in\r_{++}^{3}$.

We now present the numerical tables for this set of problem instances.

\begin{table}[H]
    \begin{centering}
        \makebox[\textwidth][c]{%
            \begin{tabular}{|>{\centering}m{0.7cm}>{\centering}m{0.7cm}>{\centering}m{0.7cm}|>{\centering}p{1.7cm}>{\centering}p{1.7cm}>{\centering}p{1.7cm}>{\centering}p{1.7cm}>{\centering}p{1.7cm}>{\centering}p{1.7cm}|}
                \hline 
                \multirow{2}{0.7cm}{\centering{}{\footnotesize{}$\beta$}} & \multirow{2}{0.7cm}{\centering{}{\footnotesize{}$\mu$}} & \multirow{2}{0.7cm}{\centering{}{\footnotesize{}$\theta$}} & \multicolumn{6}{c|}{{\small{}Iteration Count}}\tabularnewline
                &  &  & {\footnotesize{}UPFAG} & {\footnotesize{}NC-FISTA} & {\footnotesize{}AG} & {\footnotesize{}R-AIPPc} & {\footnotesize{}R-AIPPv1} & {\footnotesize{}R-AIPPv2}\tabularnewline
                \hline 
                {\footnotesize{}$1/2$} & {\footnotesize{}$\sqrt{2}$} & {\footnotesize{}$2$} & {\footnotesize{}73} & {\footnotesize{}-} & {\footnotesize{}229} & {\footnotesize{}21} & \textbf{\footnotesize{}16} & {\footnotesize{}131}\tabularnewline
                {\footnotesize{}$1$} & {\footnotesize{}$\sqrt{2}$} & {\footnotesize{}$2$} & {\footnotesize{}132} & {\footnotesize{}-} & {\footnotesize{}324} & {\footnotesize{}73} & {\footnotesize{}77} & \textbf{\footnotesize{}70}\tabularnewline
                {\footnotesize{}$2$} & {\footnotesize{}$\sqrt{2}$} & {\footnotesize{}$2$} & \textbf{\footnotesize{}76} & {\footnotesize{}-} & {\footnotesize{}-} & {\footnotesize{}210} & {\footnotesize{}356} & {\footnotesize{}83}\tabularnewline
                \hline 
        \end{tabular}}
        \par\end{centering}
    \caption{Iteration counts for BMC problems.}
\end{table}

\begin{table}[H]
    \begin{centering}
        \makebox[\textwidth][c]{%
            \begin{tabular}{|>{\centering}m{0.7cm}>{\centering}m{0.7cm}>{\centering}m{0.7cm}|>{\centering}p{1.7cm}>{\centering}p{1.7cm}>{\centering}p{1.7cm}>{\centering}p{1.7cm}>{\centering}p{1.7cm}>{\centering}p{1.7cm}|}
                \hline 
                \multirow{2}{0.7cm}{\centering{}{\footnotesize{}$\beta$}} & \multirow{2}{0.7cm}{\centering{}{\footnotesize{}$\mu$}} & \multirow{2}{0.7cm}{\centering{}{\footnotesize{}$\theta$}} & \multicolumn{6}{c|}{{\small{}Runtime (seconds)}}\tabularnewline
                &  &  & {\footnotesize{}UPFAG} & {\footnotesize{}NC-FISTA} & {\footnotesize{}AG} & {\footnotesize{}R-AIPPc} & {\footnotesize{}R-AIPPv1} & {\footnotesize{}R-AIPPv2}\tabularnewline
                \hline 
                {\footnotesize{}$1/2$} & {\footnotesize{}$\sqrt{2}$} & {\footnotesize{}$2$} & {\footnotesize{}1515.79} & {\footnotesize{}4000.00{*}} & {\footnotesize{}2498.02} & {\footnotesize{}283.48} & \textbf{\footnotesize{}254.04} & {\footnotesize{}1305.06}\tabularnewline
                {\footnotesize{}$1$} & {\footnotesize{}$\sqrt{2}$} & {\footnotesize{}$2$} & {\footnotesize{}2619.55} & {\footnotesize{}4000.00{*}} & {\footnotesize{}3754.03} & {\footnotesize{}881.60} & {\footnotesize{}900.00} & \textbf{\footnotesize{}801.55}\tabularnewline
                {\footnotesize{}$2$} & {\footnotesize{}$\sqrt{2}$} & {\footnotesize{}$2$} & {\footnotesize{}1938.81} & {\footnotesize{}4000.00{*}} & {\footnotesize{}4000.00{*}} & {\footnotesize{}2435.49} & {\footnotesize{}3657.56} & \textbf{\footnotesize{}943.33}\tabularnewline
                \hline 
        \end{tabular}}
        \par\end{centering}
    \caption{Runtimes for BMC problems.}
\end{table}

\subsection{Summary of the numerical experiments}

All three variants of the R-AIPP method perform well (relative to
the other methods) in the numerical experiments of this section. The
R-AIPPv2 method, in particular, is the best performing method in a
large proportion of both the unconstrained and constrained problem
instances. A potential explanation is that the stepsizes $\{\lam_{k}\}$
generated by this method may become significantly larger than the
initial stepsize parameters ${\lam_0}=1$ and ${\lam_0}=0.9/(2m)$
used in the R-AIPPv1 and R-AIPPc methods, respectively, which in view
of the third remark following Proposition~\ref{sec:background},
speeds up the convergence of the quantity $\min_{i\leq k}\|\hat{v}_{i}\|$
to zero.

Moreover, the adaptive stepsize R-AIPP variants, namely, the R-AIPPv1 and R-AIPPv2 methods, have been shown to perform well regardless of the size of the ratio $M/m$ (see, for example, Tables~\ref{tab:high_ratio_1}--\ref{tab:low_ratio_2}). This is a significant improvement over the AIPP method of \cite{WJRproxmet1} which has only been shown to perform well when the ratio $M/m$ is large (see, for example, Table~\ref{tbl:aipp_comp}).

\section{Concluding remarks} \label{sec:concl_remarks}

Observing the arguments used in the proofs of Proposition~\ref{prop:pen_termination}, Lemma~\ref{lem:poten_bd}, and Theorem~\ref{thm:rqp_aipp_compl}, it is straightforward to see that the assumption of $\dom h$ being bounded can be relaxed to assuming that the iterates $\{\hat{z}_l\}$ generated by R-QP-AIPP method of Section~\ref{sec:penalty} be bounded. Explicitly assuming that the iterates satisfy $\|\hat z_l\| \leq B$, for every $l\geq 1$ and some $B>0$, the resulting ACG iteration complexity of R-QP-AIPP method is \eqref{eq:r_qp_aipp_compl} with $Q_0$ replaced by the quantity
\[
\varphi_{c_{0}}(\hat{z}_{0})-\varphi^*_{c_{0}} + 2\left({\varphi}^*-\varphi^*_0+\hat{\rho}\left[d_0 +2B\right]+ \underline{m} \left[d_0^2 + 4B^{2}\right]\right),
\]
where $c_0$ is as in step~0 of the method, $d_0 := \inf \{ \|u - \hat z_0\| : z\in {\cal F} \}$, the quantity $\underline{m}$ is as in \eqref{eq:m_lower_def} with $g=f$, and the quantities $\hat z_0, \varphi_c,$ and $\varphi_c^*$ are from the input of the R-QP-AIPP method and \eqref{eq:pen_sub}.
It should be noted however that we were not able to show that the iterates $\{\hat z_l\}$ is bounded.
Hence, it is still an open problem to establish the iteration complexity of R-QP-AIPP when
$\dom h$ is unbounded.

Note that the description of the R-AIPP (resp. R-QP-AIPP) method of Section~\ref{sec:AIPPmet} (resp. Section~\ref{sec:penalty}) does not actually require knowledge of an upper bound $m$ on the parameter $\underline{m}$ in \eqref{eq:m_lower_def}. This is in contrast to the AIPP (resp. QP-AIPP) method of \cite{WJRproxmet1}, which requires $m$ in order to establish its validity and iteration complexity. In addition, one could consider a R-AIPP (resp. R-QP-AIPP) variant in which the quantity $M$ (resp. $L$) is adaptively inferred from its iterates rather than requiring knowledge of its value beforehand. While for the sake of brevity we omit the formal description and analysis of such a variant in this paper, we conjecture that the iteration complexity of the R-AIPP (resp. R-QP-AIPP) variant is as in \eqref{eq:r_aipp_compl} (resp. \eqref{eq:r_qp_aipp_compl}) with $M$ (resp. $L$) replaced with a quantity that lower bounds it, e.g., the maximum of the lower estimates of $M$ (resp. $L$) which are inferred by the generated iterates.

\appendix

\section{Appendix}

This appendix contains proofs and statements of several technical results used in the main body of the paper. It contains three subsections. The first subsection consists of proofs about the refinement procedure of Section~\ref{sec:background}; the second subsection consists of proofs about the R-ACG algorithm of Section~\ref{sec:Nesterov's-Method}; and the third subsection consists of technical results related to Section~\ref{sec:penalty}.

\subsection{Properties of the refinement procedure}

\begin{proof}[of Proposition~\ref{prop:approxsol}]
    It follows from \cite[Lemma 19]{WJRproxmet1} with $(f,h,L)=(f_{\lam},h_{\lam},M_{\lam})$
    that $\Delta\geq0$ and 
    \[
    M_{\lam}(z-\hat{z})\in\nabla f_{\lam}(z)+\pt h_{\lam}(\hat{z})=\lam\nabla g(z)+(z-z^{-}-v)+\lam\pt h(\hat{z}).
    \]
    Dividing by $\lam$ and rearranging terms yields 
    \[
    \frac{1}{\lam}\left[M_{\lam}(z-\hat{z})+(v+z^{-}-z)\right]-\nabla g(z)\in\pt h(\hat{z}).
    \]
    Adding $\nabla g(\hat{z})$ to both sides and using the definition of $\hat{v}$
    gives 
    \[
    \hat{v}=\frac{1}{\lam}\left[M_{\lam}(z-\hat{z})+(v+z^{-}-z)\right]+\nabla g(\hat{z})-\nabla g(z)\in\nabla g(\hat{z})+\pt h(\hat{z}),
    \]
    which is the inclusion in \eqref{eq:inclv'}.
    
    We now bound $\lam\|\hat{v}\|$. Since
    \cite[Lemma 19]{WJRproxmet1} implies that $\|z-\hat{z}\|\leq\sqrt{2M_{\lam}^{-1}\Delta}$
    and $\nabla g$ is $M$--Lipschitz continuous then 
    \begin{align*}
    \lam\|\hat{v}\| & \leq\|M_{\lam}(z-\hat{z})\|+\|v+z^{-}-z\|+\lam\|\nabla g(\hat{z})-\nabla g(z)\|\\
    & \leq\sqrt{2M_{\lam}\Delta}+\|v+z^{-}-z\|+\lam M\|\hat{z}-z\|
    \leq\sqrt{2M_{\lam}\Delta}+\|v+z^{-}-z\|+M_{\lam}\|\hat{z}-z\|\\
    & \leq\sqrt{2M_{\lam}\Delta}+\|v+z^{-}-z\|+M_{\lam}\sqrt{2M_{\lam}^{-1}\Delta}
    =\|v+z^{-}-z\|+2\sqrt{2M_{\lam}\Delta},
    \end{align*}
    which is the inequality in \eqref{eq:inclv'}.    
\end{proof}

\subsection{Properties of the R-ACG algorithm} \label{app:r_acg}

\begin{proof}[of Proposition~\ref{prop:nest_complex}(a)]
    Let $\ell$ denote the quantity in \eqref{eq:numbiter_ACG}. Assume that
    the R-ACG algorithm has performed $\ell$-iterations without declaring failure.
    In view of step~2 of the R-ACG algorithm, it follows that both \eqref{mainIneqAACG} and \eqref{crit:subIneqACG} hold for every $1\leq j\leq\ell$.
    We will show that it must stop successfully at the end of the $\ell^{\rm th}$ iteration, and hence that
    the conclusion of the lemma holds. Indeed,
    note that \eqref{increasingAj}, \eqref{eq:numbiter_ACG}, and the fact that $\log(1+t) \leq t$ for all $t\geq 0$ implies that
    \begin{equation}\label{eq:incAJprop1}
    A_\ell\geq \frac{2}{1+ 2\widetilde M}\left(1+ \frac{1}{2} \sqrt{\frac{1}{1+2\widetilde M}}\right)^{2(\ell-1)}\geq 2C > 2.
    \end{equation}
    Combining  the triangle inequality, \eqref{mainIneqAACG}, the fact that $2/A_\ell \leq 1/C$ and $(2/A_\ell)^2 < 2/A_\ell < 1$ from \eqref{eq:incAJprop1}, and  the relation $(a+b)^2\leq 2(a^2+b^2)$ for all $a, b \in \R$, we obtain 
    \begin{align*}
    \|u_\ell\|^2+2\eta_\ell
    &\leq \max\{1/A_\ell^2,1/(2A_\ell)\}(\|A_\ell u_\ell\|^{2}+4A_\ell\eta_\ell)\\
    &\leq  \max\{1/A_\ell^2,1/(2A_\ell)\}(2\|A_\ell u_\ell+x_\ell-x_0\|^{2}+2\|x_\ell-x_0\|^{2}+4A_\ell \eta_\ell)\\
    &\leq  \max\{(2/A_\ell)^2,2/A_\ell\}\|x_\ell-x_0\|^{2} \leq   \frac{1}{C}\|x_\ell-x_0\|^{2}.
    \end{align*}
    On the other hand, using the triangle inequality  and the fact that $(a+b)^2\leq (1+s)a^2+(1+1/s)b^2$ for every $(a,b,s)\in \R\times \R \times R_{++}$ (under the choice of $s=1/(\sqrt{C}-1)$), we obtain
    \begin{equation*}
    \|x_\ell-x_0\|^{2}\leq\frac{\sqrt{C}}{\sqrt{C}-1}\|x_0-x_\ell+u_\ell\|^2+\sqrt{C}\|u_\ell\|^2.
    \end{equation*} 
    Combining the previous  estimates, we then conclude that
    \begin{equation}
    \|u_\ell\|^2+2\eta_\ell \leq \frac{1}{C-\sqrt{C}}\|x_0-x_\ell+u_\ell\|^2 + \frac{1}{\sqrt{C}}\|u_\ell\|^2,
    \end{equation}
    which, after a simple algebraic manipulation, easily implies that 
    \begin{gather}
    \frac{1}{\sqrt{C}-1}\|x_{0}-x_{\ell}+u_{\ell}\|^{2}
    \geq 2\sqrt{C}\eta_{\ell}+\left(\sqrt{C}-1\right)\|u_{\ell}\|^{2}
    \geq\left(\sqrt{C}-1\right)\left(\|u_{\ell}\|^{2} + 2\eta_{\ell} \right). \label{gipp_sigma_bd} 
    \end{gather}
    Using the first term in the maximum of \eqref{eq:C_def} together with the second inequality of \eqref{gipp_sigma_bd} immediately implies that
    \eqref{ineq:ACGHPE_algo} holds with $j=\ell$.
    To show that \eqref{crit:ACGdescent_algo} holds at $j=\ell$, observe that the definition of $\psi$ in \eqref{eq:acg_defs}, 
    \eqref{crit:subIneqACG} with $j=\ell$, the second inequality of \eqref{gipp_sigma_bd}, and the second term in the maximum of \eqref{eq:C_def} imply that
    \begin{align*}
     \widetilde{\phi}(x_{0})-\widetilde{\phi}(x_{\ell}) & \geq\left\langle u_{\ell},x_{0}-x_{\ell}\right\rangle +\eta_{\ell}+\frac{1}{2}\|x_{\ell}-x_{0}\|^{2} 
     = \frac{1}{2}\left[\|x_{0}-x_{\ell}+u_{\ell}\|^{2}-\left(\|u_{\ell}\|^{2}+2\eta_{\ell}\right)\right] \\
     & \geq \frac{1}{2} \left[1 + \left(\sqrt{C}-1\right)^{-2}\right]\|x_{0}-x_{\ell}+u_{\ell}\|^{2} \geq\frac{1}{\theta}\|x_{0}-x_{\ell}+u_{\ell}\|^{2}.
    \end{align*}
\end{proof}

\subsection{Results related to Section~\ref{sec:penalty}} \label{app:penalty}

\begin{lemma}\label{lem:lsc}
    Assume that $f, h:\rn\mapsto(-\infty,\infty]$ 
    satisfy assumptions (C1) and (C3) in Section~\ref{sec:penalty}, and that,
    in addition, $f$ is lower semicontinuous on $\cl(\dom h)$. Then,
    $\varphi:=f+h$ is a  proper lower semicontinuous function which has
     a global minimum over $\R^n$.
\end{lemma}

\begin{proof}
    Suppose $\bar{z}\in\rn\backslash\cl(\dom h)$.
    Since $\cl(\dom h)$ is
    closed, there exists $\varepsilon>0$ such that $h(u) = \infty$ for every $u \in \rn\backslash\cl(\dom h)$ satisfying $\|u-\bar z\| < \varepsilon$.
    Hence, $\liminf_{u\to\bar{z}}\varphi(u)=\infty=\varphi(\bar{z})$. Now
    suppose $\bar{z}\in\cl(\dom h)$. By the lower semicontinuity of $f$
    and $h$ we have 
    \[
    \liminf_{u\to\bar{z}}\varphi(u)\geq\liminf_{u\to\bar{z}}f(u)+\liminf_{u\to\bar{z}}h(u)\geq f(\bar{z})+h(\bar{z})=\varphi(\bar{z})
    \]
    and, since $f$ is differentiable on $\dom h$, the function $\varphi$ is proper lower semicontinuous with $\dom \varphi = \dom h$. The last statement of the lemma follows from the well known fact that infimum of a lower semicontinuous function over a bounded set, namely, $\dom \varphi$, is always attained.
\end{proof}

\subsection{Comparison with the AIPP method}\label{app:comp_aipp}

This subsection presents some computational results that compare the AIPP method of \cite{WJRproxmet1} with the R-AIPPc method described at the beginning of Section~\ref{sec:numerical}. The main problem of interest for this sub-subsection is the quadratic matrix problem described in Sub-subsection~\ref{subsubsec:qm_prb}.

We now describe the particular implementation of the AIPP method used in this sub-subsection, which differs from its description in \cite{WJRproxmet1} in two ways. First, its innermost subroutine, namely, the ACG method, stops immediately when a quadruple $(\lam_k, z_k, v_k, \varepsilon_k)$ satisfying \eqref{eq:gipp_main} is found. Second, for each iteration $k$ of the method, a triple $(\hat z, \hat v, \Delta)$ is generated from the refinement procedure in Section~2 by assigning $(\hat z, \hat v, \Delta) = RP(\lam_k, z_{k-1}, z_k, v_k)$, and the method stops with the desired output when $\hat v$ satisfies condition \eqref{eq:term_unconstr}.

All experiment parameters for the R-AIPPc method and the problem instances are as described in Sub-subsection~\ref{subsubsec:qm_prb} below, while the AIPP uses a parameter input of $(\sigma,\lam)=(0.3, 1/(2m))$ for its results.

We now present the numerical tables for this set of problem instances.

\begin{table}[H]
    \begin{centering}
        \makebox[\textwidth][c]{%
            \begin{tabular}{|>{\centering}m{0.7cm}>{\centering}m{0.7cm}|>{\centering}m{1.8cm}|>{\centering}p{1.7cm}>{\centering}p{1.7cm}|>{\centering}p{1.7cm}>{\centering}p{1.7cm}|}
                \hline 
                \multirow{2}{0.7cm}{\centering{}{\footnotesize{}$M$}} & \multirow{2}{0.7cm}{\centering{}{\footnotesize{}$m$}} & \multirow{2}{1.8cm}{\centering{}{\footnotesize{}$\phi(\hat{z})$}} & \multicolumn{2}{c|}{{\small{}Iteration Count}} & \multicolumn{2}{c|}{{\footnotesize{}Runtime}}\tabularnewline
                &  &  & {\footnotesize{}AIPP} & {\footnotesize{}R-AIPPc} & {\footnotesize{}AIPP} & {\footnotesize{}R-AIPPc}\tabularnewline
                \hline 
                {\footnotesize{}$10^{1}$} & {\footnotesize{}$10^{0}$} & {\footnotesize{}-3.65E-02} & {\footnotesize{}-} & \textbf{\footnotesize{}8920} & {\footnotesize{}4000.00{*}} & \textbf{\footnotesize{}1098.14}\tabularnewline
                {\footnotesize{}$10^{2}$} & {\footnotesize{}$10^{0}$} & {\footnotesize{}-1.74E-02} & {\footnotesize{}53276} & \textbf{\footnotesize{}8317} & {\footnotesize{}3672.23} & \textbf{\footnotesize{}737.76}\tabularnewline
                {\footnotesize{}$10^{3}$} & {\footnotesize{}$10^{0}$} & {\footnotesize{}2.05E-02} & {\footnotesize{}23645} & \textbf{\footnotesize{}4424} & {\footnotesize{}1547.97} & \textbf{\footnotesize{}329.06}\tabularnewline
                {\footnotesize{}$10^{4}$} & {\footnotesize{}$10^{0}$} & {\footnotesize{}3.67E-01} & {\footnotesize{}7797} & \textbf{\footnotesize{}3250} & {\footnotesize{}505.20} & \textbf{\footnotesize{}215.69}\tabularnewline
                {\footnotesize{}$10^{5}$} & {\footnotesize{}$10^{0}$} & {\footnotesize{}3.82E+00} & {\footnotesize{}2791} & \textbf{\footnotesize{}1420} & {\footnotesize{}176.64} & \textbf{\footnotesize{}93.87}\tabularnewline
                {\footnotesize{}$10^{6}$} & {\footnotesize{}$10^{0}$} & {\footnotesize{}3.84E+01} & {\footnotesize{}3475} & \textbf{\footnotesize{}1270} & {\footnotesize{}222.26} & \textbf{\footnotesize{}84.84}\tabularnewline
                \hline 
        \end{tabular}}
        \par\end{centering}
    \caption{Iteration counts for QM problems with fixed $m$.}
    \label{tbl:aipp_comp}
\end{table}

\bibliographystyle{plain}
\bibliography{Proxacc_ref}

\end{document}